\documentclass[12pt]{amsart}
\usepackage{amscd,amssymb,amsopn,amsmath,amsthm,graphics,amsfonts,enumerate,verbatim,calc,bbm}
\usepackage{mathrsfs}
\usepackage{tikz}
\usepackage[colorlinks=true,linkcolor=red,citecolor=blue]{hyperref}
\usepackage[top=2.8cm, bottom=2.8cm, left=2.54cm, right=2.54cm]{geometry}

\newtheorem{thm}{Theorem}[section]
\newtheorem{cor}[thm]{Corollary}
\newtheorem{lem}[thm]{Lemma}

\newtheorem{prop}[thm]{Proposition}
\theoremstyle{definition}
\newtheorem{exam}[thm]{Example}

\newtheorem{rem}[thm]{Remark}
\newtheorem{defn}[thm]{Definition}

\numberwithin{equation}{section}
\let\ds=\displaystyle
\newcommand{\st}{\,:\,}

\newcommand{\mr}{\mathop{\mathrm{mr}}}
\newcommand{\mgr}{\mathop{\mathrm{mgr}}}
\newcommand{\mc}{\mathop{\mathrm{mc}}}

\newcommand{\norm}[1]{\left \Vert #1 \right \Vert}

\allowdisplaybreaks
\begin{document}

\title{Roundness Properties of Banach spaces}

\author[A. Amini-Harandi]{Alireza Amini-Harandi$^{1,2}$}

\address{$^{1}$Department of Pure Mathematics, Faculty of Mathematics and Statistics, University of Isfahan, Isfahan 81746-73441, Iran}
\address{$^{2}$School of Mathematics, Institute for Research in Fundamental Sciences (IPM), P. O. Box: 19395-5746, Tehran, Iran}
\email{a.amini@sci.ui.ac.ir}

\author[I. Doust]{Ian Doust$^{3}$}

\address{$^{3}$School of Mathematics and Statistics, University of New South Wales, Sydney, NSW 2052, Australia}
\email{i.doust@unsw.edu.au}

\author[G. Robertson]{Gavin Robertson$^{3}$}

\email{gavin.robertson@unsw.edu.au}

\begin{abstract}
The maximal roundness of a metric space is a quantity that arose in the study of embeddings and renormings. In the setting of Banach spaces, it was shown by Enflo that roundness takes on a much simpler form. In this paper we provide simple computations of the roundness of many standard Banach spaces, such as $\ell^{p}$, the Lebesgue-Bochner spaces $\ell^{p}(\ell^{q})$ and the Schatten ideals $S_{p}$. We also introduce a property that is dual to that of roundness, which we call coroundness, and make explicit the relation of these properties to the geometric concepts of smoothness and convexity of Banach spaces. Building off the work of Enflo, we are then able to provide multiple non-trivial equivalent conditions for a Banach space to possess maximal roundness greater than $1$. Using these conditions, we are able to conclude that certain Orlicz spaces possess non-trivial values of roundness and coroundness. Finally, we also use these conditions to provide an explicit example of a $2$-dimensional Banach space whose maximal roundness is not equal to that of its dual.
\end{abstract}


\maketitle

\section{Introduction}\label{Introduction}

The concept of roundness of a metric space was first introduced by Enflo in the $1960$s in his papers on the infinite dimensional version of Hilbert's 5th problem. In particular, he used roundness to deal with the remaining cases of the result that  no infinite dimensional $L^p$ space is uniformly homeomorphic to any $L^q$ space (with $1 \le p,q$ and $p \ne q$) \cite{Lind,E1,E2}. However, this use of roundness in the setting of uniform homeomorphisms is rather misleading about the nature of roundness, since it is actually an isometric invariant of metric spaces rather than a uniform one. It should be noted that roundness was an important precursor to several later concepts, including Rademacher type and cotype of Banach spaces and generalised roundness of metric spaces. In fact, in the setting of Banach spaces roundness is actually a special case of Rademacher type (see Remark \ref{Type and Cotype}). Enflo's use of roundness in the uniform setting can then be explained by a celebrated rigidity theorem of Ribe \cite{R}, which shows that Rademacher type of Banach spaces is invariant under uniform homeomorphisms (even though roundness is not).

The definition of roundness as given by Enflo in \cite{E1} is as follows.

\begin{defn}
	A metric space $(X,d)$ has \textit{roundness} $p\geq 1$, denoted $p\in r(X,d)$, if whenever $a_1,a_2,b_1,b_2$ are in $X$ we have
	\begin{align*}
		d(a_1,a_2)^p+d(b_1,b_2)^p\leq \sum_{1\leq i,j\leq 2}  d(a_i,b_j)^p.
	\end{align*}
\end{defn}

A crucial result is that in the case of a Banach space $(X,\|\cdot\|)$, the above definition takes on a particularly simple form. Enflo \cite{E3} showed that
\begin{itemize}
	\item [(i)] $p\in r(X,\|.\|)$ if and only if
	\[
	\|x+y\|^p+\|x-y\|^p \leq 2(\|x\|^p+\|y\|^p)
	\]
	for all $x,y\in X$,
	\item [(ii)] $p\in r(X,\|.\|)$ implies $p_1\in r(X,\|.\|)$ for all $1 \leq p_1 \leq p$.
\end{itemize}
We remark that the second property above does not hold in general for arbitrary metric spaces (see \cite{E3}). For a Banach space however, by this second property it follows that the study of roundness of the space space really boils down to a study of its maximal roundness. That is, the roundness properties of a Banach space are entirely determined by the quantity
\[\mr(X) = \sup\{p \st \text{$p \in r(X,\norm{\cdot})$}\} .\]
Let us now note some basic properties of $\mr(X)$. Firstly, if $X$ is a Banach space, then $1 \le \mr(X) \le 2$. Secondly, if $X$ and $Y$ are Banach spaces and $X$ embeds isometrically in $Y$ then $\mr(Y) \le \mr(X)$. It is also not too hard to see that $\mr(X) = 2$ if and only if $X$ is a Hilbert space (see \cite{Cl}). Harder is a classification of those spaces for which one obtains the other extreme value $\mr(X) = 1$ (see Section \ref{Nontrivial Roundness} for a classification of such spaces).

Enflo's work surrounding the infinite dimensional version of Hilbert's 5th problem made it clear that the property of having $\mr(X)>1$ was a desirable one. Indeed, many of the important structure theorems involving roundness only apply in the case that the space in question has roundness strictly greater than one. For example, if a locally bounded topological linear space is uniformly homeomorphic to a Banach space with roundness greater than one, then it is a normable space (see \cite{E4} for this and some related results). However, since the property of having $\mr(X)>1$ was sufficient for many of these theorems, the question of how one may explicitly compute the precise value of $\mr(X)$ for a given Banach space did not really feature in Enflo's work.

Following the results of Enflo on the roundness of Banach spaces, the concept of roundness has been studied for other special classes of metric spaces. In \cite{LP} Lafont and Prassidis studied the roundness of two different classes of metric spaces. The first class that they studied was that of geodesic manifolds, where they were able to provide a connection between the topological properties of the space and the values of roundness of the space. For example, they proved that all geodesic spaces of roundness $2$ are contractible, and that every compact Riemannian manifold with roundness greater than $1$ must be simply connected. The second class of metric spaces that they studied was that of Cayley graphs of finitely generated groups. Here they were able to provide connections between the algebraic properties of the group with roundness properties of the Cayley graph. For example, they proved that if every Cayley graph of a group has roundness greater than $1$ then the group must be a torsion group with every element of order $2,3,5$ or $7$. In \cite{HLMRR} Horak, LaRose, Moore, Rooney and Rosenthal studied the more general case of roundness properties of finite connected graphs. Here they were able to explicitly compute the roundness of certain families of graphs, such as the cycles $C_{n}$. They were also able to show that in this setting certain restrictions on the roundness of all finite connected graphs prevailed. For example, no finite connected graph can have roundness strictly between $\log_{2}(3)$ and $2$.

It is worthwhile to note that a theory that runs parallel to the theory of roundness, and which heavily motivates many of the ideas and questions pursued in this paper, is the theory of generalised roundness. For the sake of completeness, we state the definition of generalised roundness here, as originally defined by Enflo in \cite{E2}.

\begin{defn}
	A metric space $(X,d)$ has \textit{generalised roundness} $p\geq 0$, denoted $p\in gr(X,d)$, if whenever $a_{1},\dots,a_{n},b_{1},\dots,b_{n}$ are in $X$ and $n\geq 2$ we have
	\begin{align*}
		\sum_{1\leq i,j\leq n}d(a_{i},a_{j})^{p}+d(b_{i},b_{j})^{p}\leq 2\sum_{1\leq i,j\leq n}d(a_{i},b_{j})^{p}.
	\end{align*}
\end{defn}

By taking $n=2$ in the above definition it is easy to see that for $p\geq 1$ if a metric space has generalised roundness $p$ then it also has roundness $p$. In recent years, the concept of generalised roundness of metric spaces has been the subject of a large amount of research activity (see for example \cite{Do,Do2,Fa,LTW,Mu,Ro,S}). One of the main reasons for this is perhaps the significant embedding implications that generalised roundness carries with it, the main theorem in this area being that a metric space $(X,d)$ has generalised roundness $p$ if and only if $(X,d^{p/2})$ isometrically embeds in a Hilbert space (see \cite{Fa} for a proof of this result). A nontrivial fact is that, just as for roundness of Banach spaces, the theory of generalised roundness of metric spaces boils down to the computation of the maximal generalised roundness of a space, denoted by $\mgr(X)=\sup\{p\geq 0:p\in gr(X,d)\}$.

The computation of $\mgr(X)$ for any particular space is challenging, and has only been done for a rather small collection of spaces. Even for finite metric spaces the problem is usually analytically intractable.
However, S\'{a}nchez in \cite{S} provided a formula for $\mgr(X)$ that allowed for the numerical computation of $\mgr(X)$ for an arbitrary finite metric space. Let us briefly describe this formula now. Suppose that $X=\{x_{0},x_{1},\dots,x_{n}\}$ is a finite metric space, with distance matrix $D = (d(x_i,x_j))_{i,j=0}^n$. For $p>0$, let $D_p$ denote the `$p$-distance matrix', $(d(x_i,x_j)^p)_{i,j=0}^n$. Also, let $\mathbbm{1}$ denote the column vector in $\mathbb{R}^{n+1}$ all of whose coordinates are $1$, and let $\langle\cdot,\cdot\rangle$ denote the standard inner product on $\mathbb{R}^{n+1}$. S{\'a}nchez \cite{S} showed that if $\mgr(X)<\infty$ then
	\[ \mgr(X) = \min\{p>0 \,:\, \text{$\det(D_p) = 0$ or $\langle D_{p}^{-1}\mathbbm{1},\mathbbm{1}\rangle$}=0\}.
	\]
For several special families of metric spaces, such as the bipartite graphs $K_{m,n}$, this formula even allows for a precise (algebraic) computation of the maximal generalised roundness of the space (see \cite{S}). As such, since the conception of this formula, the state of knowledge pertaining to the maximal generalised roundness of finite metric spaces has improved. For explicit results obtained using this formula, see \cite{Do,Fa,Mu}. A related formula is given in \cite{Ro}.
	
In this article, the class of metric spaces whose roundness properties we will study is that of Banach spaces. The main motivating question for much of the work below is the question of how one may actually compute the quantity $\mr(X)$ for a given Banach space. However, as was the case with the maximal generalised roundness of metric spaces (at least, before the above formula of S\'{a}nchez), this is in general an extremely difficult task. Of course, for certain Banach spaces such as $L^{p}$ spaces and other similar spaces, a computation of $\mr(X)$ can be achieved through classical inequalities such as Minkowski's Inequality and Clarkson's Inequalities and so on. However, in the grand scheme of things the list of Banach spaces for which such an approach is feasible is extremely short. Even for $2$-dimensional Banach spaces, the task of determining whether or not a space has a given value of roundness may be impossible (see for example the $2$-dimensional Banach spaces in Section \ref{Further Examples}). In fact, even proving that $\mr(X)>1$ is in general an extremely difficult task. Thus, it seems that the best that one can hope for is perhaps a formula analogous to that of S\'{a}nchez' formula in the setting of roundness of Banach spaces. In this article, we take the first steps towards this goal, and focus on developing a systematic approach to determining whether or not $\mr(X)>1$ for a given Banach space.

\subsection{Overview of the article}

In Section \ref{The Modulus Of Roundness} we introduce a geometric function related to the theory of roundness, and derive many of its basic properties. In doing so, we obtain new simple proofs of many of the known elementary properties of roundness. In fact, this section makes it apparent that the theory of roundness of Banach spaces is really just a small part of the theory of these geometric functions.

As stated earlier, the computation of the maximal roundness of certain spaces is accessible through classical inequalities. Even so, there does not seem to be a unified approach to this task for many of these standard spaces. To address this problem we introduce a number of new concepts. The first, which is suggested by Clarkson's inequalities for $L^p$ spaces, is a variant we call Clarkson roundness. Using this we give a complete listing of the maximal roundness for all the Lebesgue--Bochner spaces $\ell^{p}(\ell^{q})$ and the Schatten ideals $S_p$, extending the known results for $L^p$ spaces. This is the content of Section \ref{Roundness Of Standard Spaces}.

In Section \ref{Duality and Coroundness} we introduce a second new concept, a property that is dual to that of roundness, which we call coroundness. It turns out that the roundness and coroundness of a Banach space greatly affect its geometry. We will show, for example, that if a Banach space $X$ has maximal roundness greater than 1 then it is uniformly smooth, and if it has finite minimal coroundness then it is uniformly convex.

Expanding upon results of Enflo concerning the maximal roundness of a Banach space, in Section \ref{Nontrivial Roundness} we focus on the problem of characterising those spaces for which $\mr(X)=1$. Here we provide multiple new characterisations of when $\mr(X)=1$ in terms of many other Banach space properties such as $p$-uniform smoothness and the modulus of uniform smoothness.

Finally, in Section \ref{Further Examples} we employ the results of the previous section to provide new and interesting examples of Banach spaces that possess non-trivial values of roundness and coroundness. We are able to provide two explicit Orlicz spaces (both different from $L^{p}$ spaces), one having non-trivial roundness and the other having non-trivial coroundness. For all the standard spaces for which the maximal roundness is known, it is noticeable that the space and its dual have the same maximal roundness. The final example in Section \ref{Further Examples} is a $2$-dimensional Banach space for which the methods of the previous section allow us to conclude that the maximal roundness of this space is not equal to that of its dual.

\subsection{Notation}

In what follows, $X$ is always assumed to be a real Banach space. If $X$ is a Banach space then $B_{X}=\{x\in X:\|x\|\leq 1\}$ will be used to denote the closed unit ball in $X$, and $S_{X}=\{x\in X:\|x\|=1\}$ the unit sphere.

If $p>1$ then $p'$ will be used to denote the conjugate index satisfying $\frac{1}{p}+\frac{1}{p'}=1$ (we will also at times use the convention that if $p=1$ then $p'=\infty$ and vice versa).

If $f,g$ are real-valued functions defined on some subset of $\mathbb{R}$ then for $a\in\mathbb{R}$ we will say that $f=O(g)$ as $t\rightarrow a$ to mean that there exists some $\epsilon>0$ and $C>0$ such that $|f(t)|\leq C|g(t)|$ for all $t\in(a-\epsilon,a+\epsilon)\cap\mbox{domain}(f)\cap\mbox{domain}(g)$. If $h$ is also a real-valued function defined on some subset of $\mathbb{R}$ then we will say that $f=g+O(h)$ as $t\rightarrow a$ to mean that $f-g=O(h)$ as $t\rightarrow a$. Finally, we will say that $f=\Theta(g)$ as $t\rightarrow a$ if both $f=O(g)$ as $t\rightarrow a$ and $g=O(f)$ as $t\rightarrow a$.

\section{The modulus of roundness}\label{The Modulus Of Roundness}

To fully understand the linear theory of roundness, it is helpful to consider a certain geometric function that we associate with each Banach space. The function we consider resembles generalized von~Neumann--Jordan constants such as $C_{NJ}^{(p)}(X)$ and ${\tilde C}_{NJ}^{(p)}(X)$, as seen in \cite{CHHK} and \cite{YW2}. For our purposes however, it is more useful to consider slightly different scalings of these functions.

\begin{defn}
	Let $(X,\|\cdot\|_{X})$ be a Banach space. We define the \textit{modulus of roundness} of $X$ to be the function $\nu_{X}:[1,\infty)\rightarrow\mathbb{R}$ given by
		\[ \nu_X(p) = \sup\left\{ \frac{\|x+y\|^{p}+ \|x-y\|^{p}}{\|x\|^{p}+ \|y\|^{p}}
		\st \text{$x, y \in X$ not both zero} \right\}.
		\]
\end{defn}
It is not too hard to see that for any $1\leq p<\infty$ one has that
\begin{equation}
\max\{2,2^{p-1}\} \le \nu_X(p) \le 2^p.
\end{equation}
Other immediate consequences of the definition of $\nu_{X}$ are the following.
\begin{itemize}
 \item $\nu_X(1) = 2$ for any Banach space $X$.
 \item $p \in r(X,\norm{\cdot}) \iff \nu_X(p) \le 2 \iff \nu(p) = 2$.
 \item $\mr(X) = \sup\{p \st \nu_X(p) = 2\}$.
 \item $\nu_X(2) = 2\, C_{NJ}(X)$ (see \cite{Cl}).
\end{itemize}

In fact, it is also immediate from the definition of $\nu_{X}$ that $\nu_X(p) = 2^{p-1} C_{NJ}^{(p)}(X)$ (see \cite{CHHK} and \cite{YW2} for the definition of $C_{NJ}^{(p)}(X)$).

As is noted in, for example, \cite{KT2}, it is often helpful to consider such functions in terms of the norm of a suitable operator. For $1 \le p\leq\infty$, let $\ell_2^p(X)$ be $X \times X$ with the norm $\norm{(x,y)}_p = \bigl(\norm{x}^p + \norm{y}^p\bigr)^{1/p}$ (where for $p=\infty$ we mean $\norm{(x,y)}_{\infty}=\max(\norm{x},\norm{y})$). Also, define $T_X:X\times X\rightarrow X\times X$ to have the operator matrix
  \[ T_X = \begin{pmatrix} I_X & I_X \\ I_X & -I_X \end{pmatrix} \]
where here $I_{X}$ denotes the identity operator on $X$. That is, if $(x,y)\in X\times X$ then $T(x,y)=(x+y,x-y)$. In this notation it is easy to see that
	\[ \nu_{X}(p)=\|T_{X}\|^{p}_{\ell^{p}_{2}(X)\rightarrow \ell^{p}_{2}(X)}.
	\]
That is, $\nu_{X}(p)^{1/p}$ is precisely the operator norm of $T_{X}$, when $T_{X}$ is viewed as an operator $T_{X}:\ell^{p}_{2}(X)\rightarrow\ell^{p}_{2}(X)$. The fact that $\ell_2^p(X)^* = \ell_2^{p'}(X^*)$ and $T_{X}^* = T_{X^*}$ allows one to conclude the following duality result pertaining to the modulus of roundness.

\begin{lem}\label{L2-1}
Let $X$ be a Banach space. Then for all $p > 1$
  \[ \nu_X(p)^{1/p} = \nu_{X^*}(p')^{1/p'}.
  \]
\end{lem}
\begin{proof}
This follows from the fact that both $T_{X}$ and its adjoint $T^{\ast}_{X}=T_{X^{\ast}}$ have the same norm. Using this one simply calculates that
  \[ \nu_X(p)^{1/p} = \norm{T_X}_{\ell^{p}_{2}(X)\rightarrow\ell^{p}_{2}(X)} = \norm{T_{X^*}}_{\ell^{p'}_{2}(X^{\ast})\rightarrow\ell^{p'}_{2}(X^{\ast})} = \nu_{X^*}(p')^{1/p'}. \]
\end{proof}

Note in particular that by applying Lemma \ref{L2-1} twice, one finds that it is always the case that $\nu_{X^{\ast\ast}}=\nu_{X}$.

Another benefit of the operator theoretic reformulation of the modulus of roundness is that it allows one to conclude certain non-trivial log-convexity properties about it.

\begin{prop}\label{Nu Convexity}
	Let $X$ be a Banach space. Also, suppose that $1\leq p_{0},p_{1},p<\infty$ and $0<\theta<1$ with
		\[ \frac{1}{p}=\frac{1-\theta}{p_{0}}+\frac{\theta}{p_{1}}.
		\]
	Then
		\[ \nu_{X}(p)^{1/p}\leq\big(\nu_{X}(p_{0})^{1/p_{0}}\big)^{1-\theta}\big(\nu_{X}(p_{1})^{1/p_{1}}\big)^{\theta}.
		\]
\end{prop}
\begin{proof}
	We just make use of the fact that for $p\geq 1$ one has that
		\[ \nu_{X}(p)^{1/p}=\|T_{X}\|_{\ell^{p}_{2}(X)\rightarrow \ell^{p}_{2}(X)}.
		\]
	Since the spaces $\ell^{p}_{2}(X)$ for $1\leq p<\infty$ form an interpolation scale, the desired inequality now follows from standard results in interpolation theory (see for example \cite{BL}, Theorem 5.2.3).
\end{proof}

In particular, this gives the following.

\begin{cor}\label{T42}
Let $X$ be a Banach space. Then the  function $\nu_X$ is continuous on $[1, \infty)$, and differentiable almost everywhere in that range. Furthermore, $\nu_X$ has a constant value of 2 on $[1, \mr(X)]$ and is strictly increasing on $[\mr(X), \infty)$.
\end{cor}

\begin{proof}
	The inequality in Proposition \ref{Nu Convexity} is equivalent to the statement that the function $p\mapsto \nu_{X}(1/p)^{p}$ is log-convex, and hence this function is also convex. Since convex functions are continuous and differentiable almost everywhere, it follows from this that $\nu_{X}$ is also continuous and differentible almost everywhere on $[1,\infty)$. For the next part, suppose that $1<q<p<\infty$. Then we may apply Proposition \ref{Nu Convexity} with $\tilde{p}_{0}=1,\tilde{p}_{1}=p,\tilde{p}=q$ and $\tilde{\theta}=p(q-1)/q(p-1)$. Since $\nu_{X}(1)=2$ this gives that
		\[ \nu_{X}(q)\leq 2^{\frac{p-q}{p-1}}\nu_{X}(p)^{\frac{q-1}{p-1}}.
		\]
	So, if we now assume further that $\mr(X)\leq q<p<\infty$ then $\nu_{X}(p)>2$ and so
		\[ \nu_{X}(q)\leq 2^{\frac{p-q}{p-1}}\nu_{X}(p)^{\frac{q-1}{p-1}}< \nu_{X}(p)^{\frac{p-q}{p-1}}\times\nu_{X}(p)^{\frac{q-1}{p-1}}=\nu_{X}(p).
		\]
	That is, $\nu_{X}$ is strictly increasing on $[\mr(X), \infty)$.
\end{proof}

Formulas for $\nu_{X}$ for the $\ell^{p}$ spaces are essentially due to Kato \cite{K1}.

\begin{exam}\label{Hilbert Space}
	Let $H$ be a Hilbert space (of dimension at least 2). Then Theorem 1 of \cite{K1} shows that
	$ \nu_H(p) = \max\{2,2^{p-1}\}$.
\end{exam}

\begin{exam}\label{lp Mu}
Let $X=\ell^{p}$, and first suppose that $1\leq p\leq 2$. Then
	\[ \nu_{X}(q) = \begin{cases}
	2, & 1 \leq q \leq p, \\
	2^{q/p}, &  p\leq q\leq p', \\
	2^{q-1}, & p'\leq q<\infty.
	\end{cases}
	\]
Now suppose that $X=\ell^{p}$ where $2\leq p\leq \infty$. Then
	\[ \nu_{X}(q) = \begin{cases}
	2, & 1 \leq q \leq p', \\
	2^{q/p'}, &  p'\leq q\leq p, \\
	2^{q-1}, & p\leq q<\infty.
	\end{cases}
	\]
\end{exam}

\section{Clarkson roundness and the roundness of standard spaces}\label{Roundness Of Standard Spaces}

In order to calculate the maximal roundness of the Lebesgue--Bochner spaces $\ell^{p}(\ell^{q})$ and the Schatten spaces $S_p$ it is convenient to define a second, stronger version of roundness, whose definition is motivated by Clarkson's inequalities for $L^p$ spaces (see for example \cite{Ca}, page 119).

\begin{defn}
 For $1 < p \le 2$ we shall say that a Banach space $X$ has \textit{Clarkson roundness $p$} if
   \[ \norm{x+y}^{p'} + \norm{x-y}^{p'} \le 2 \bigl(\norm{x}^p + \norm{y}^p \bigr)^{p'-1}
   \]
for all $x,y \in X$.
\end{defn}

\begin{rem}
	The above inequality can be rearranged to read as
		\[ \bigl(\norm{x+y}^{p'} + \norm{x-y}^{p'}\bigr)^{1/p'} \le 2^{1/p'} \bigl(\norm{x}^p + \norm{y}^p \bigr)^{1/p}
		\]
	for all $x,y\in X$. Using the operator theoretic notation from Section \ref{The Modulus Of Roundness} it is then clear that a Banach space $X$ has Clarkson roundness $p$ if and only if
		\[ \|T_{X}\|_{\ell^{p}_{2}(X)\rightarrow\ell^{p'}_{2}(X)}\leq 2^{1/p'}.
		\]
	Such inequalities have been studied previously. See for example \cite{HKT}, where they were referred to as $(p,p')$-Clarkson inequalities.
\end{rem}

There is a simple relationship between roundness and Clarkson roundness.

\begin{lem}\label{CRR} Suppose that $1 < p \le 2$. Then if $X$ has Clarkson roundness $p$ it also has roundness $p$.
\end{lem}
\begin{proof}
We make use of the following elementary inequality. If $1\leq r<s<\infty$ and $a,b\geq 0$ then
	\[ (a^{s}+b^{s})^{1/s}\leq(a^{r}+b^{r})^{1/r}\leq 2^{1/r-1/s}(a^{s}+b^{s})^{1/s}.
	\]
Since $p\leq p' $, for any $x,y\in X$ this gives that
\begin{align*}
\norm{x+y}^p+\norm{x-y}^p
&\leq 2^{1-p/p' }(\norm{x+y}^{p' }+\norm{x-y}^{p' })^{p/p' }\\
&=2^{1-p/p' }(\norm{x+y}^{p' }+\norm{x-y}^{p' })^{p-1}\\
&\leq 2^{1-p/p' }\times 2^{p/p' }(\norm{x}^p+\norm{y}^p)^{(p' -1)(p-1)}\\
&=2(\norm{x}^p+\norm{y}^p)
\end{align*}
where this last equality follows from the fact that $(p'-1)(p-1)=1$.
\end{proof}

\begin{thm}\label{LB}
Let $(\Omega,\mathcal{A},\mu)$ be a measure space and let $X$ be a Banach space. Then,
 \begin{itemize}
	\item[(i)]  If $1<p\leq 2$ and $X$ has Clarkson roundness $q\leq p$, the Lebesuge--Bochner space $L^{p}(\Omega,\mathcal{A},\mu;X)$ has Clarkson roundness $q$.
	\item [(ii)] If $2\leq p<\infty$ and $X$ has Clarkson roundness $q\leq p' $, the Lebesgue--Bochner space $L^{p}(\Omega,\mathcal{A},\mu;X)$ has Clarkson roundness $q$.
\end{itemize}
\end{thm}
\begin{proof}
We begin by recognising some helpful identities. Namely that
	\[ \norm{f}_{p}^q=\bigg(\int_{\Omega}\norm{f(t)}^pd\mu(t)\bigg)^{q/p}=\big\|\norm{f}^q\big\|_{p/q}
	\]
where here $\norm{f}$ denotes the pointwise norm of a function $f:\Omega\rightarrow X$ and $\|.\|_{p}$ denotes the norm in $L^{p}(\Omega,\mathcal{A},\mu;X)$. Similarly, by swapping $q$ to $q' $ we also have
	\[\norm{f}_{p}^{q' }=\big\|\norm{f}^{q' }\big\|_{p/q' }.
	\]
Note that for $p$ and $q$ in either case described in the statement of the theorem we have that $p/q' \leq 1$ and $p/q\geq 1$. Hence, for all $f,g\in L^{p}(\Omega,\mathcal{A},\mu;X)$,
\begin{align*}
\norm{f+g}_{p}^{q' }+\norm{f-g}_{p}^{q' }
&=\big\|\norm{f+g}^{q' }\big\|_{p/q' }+\big\|\norm{f-g}^{q' }\big\|_{p/q' }\\
&\leq \big\|\norm{f+g}^{q' }+\norm{f-g}^{q' }\big\|_{p/q' }\\
&\leq 2\big\|(\norm{f}^q+\norm{g}^q)^{q' -1}\big\|_{p/q' }\\
&=2\bigg(\big\|\norm{f}^q+\norm{g}^q\big\|_{p/q}\bigg)^{q' -1}\\
&\leq 2\bigg(\big\|\norm{f}^q\big\|_{p/q}+\big\|\norm{g}^q\big\|_{p/q}\bigg)^{q' -1}\\
&=2(\norm{f}_{p}^q+\norm{g}_{p}^q)^{q' -1}
\end{align*}
where the first inequality follows from the reverse Minkowski inequality since $p/q' \leq 1$, the second inequality comes from the fact that $X$ has Clarkson roundness $q$ and the last inequality follows from the usual Minkowski inequality since $p/q\geq 1$.
\end{proof}

To make use of the above theorem, we first note the Clarkson roundness of $\mathbb{R}$. A proof of the following proposition can be found on page 118 of \cite{Ca}.

\begin{prop}\label{CR}
The Banach space $\mathbb{R}$ has Clarkson roundness $p$ for all $1<p\leq 2$.
\end{prop}

\begin{cor}\label{CL}
Let $1\leq p<\infty$. Then $\ell^p$ has maximal roundness $\min(p,p' )$.
\end{cor}
\begin{proof}
Suppose that $1<p<\infty$. By Proposition \ref{CR}, an application of Theorem \ref{LB} with $X=\mathbb{R}$ shows that $\ell^p$ has Clarkson roundness $\min(p,p' )$. Hence by Lemma \ref{CRR} $\ell^p$ also has roundness $\min(p,p' )$. At this point we note that this is also trivially true for $p=1$. To see that this is in fact the maximal roundness, one need only look to the bounds given by taking first $x=(1,0,0,0,\dots,),y=(0,1,0,0,\dots,)\in\ell^p$ and secondly $x=(1,1,0,0,\dots,),y=(1,-1,0,0,\dots,)\in\ell^p$.
\end{proof}

A second application of Theorem \ref{LB} with $X=\ell^q$ gives the following.

\begin{cor}\label{Double Lp Roundness}
Let $1\leq p,q<\infty$. Then $\ell^p(\ell^q)$ has maximal roundness $\min(p,p' ,q,q' )$.
\end{cor}
\begin{proof}
Firstly, if either $p$ or $q$ is $1$, then $\ell^p(\ell^q)$ contains an isometric copy of $\ell^1$, so that the maximal roundness of $\ell^p(\ell^q)$ must be $1$. Otherwise $1<p,q<\infty$. Then, as in the proof of Corollary \ref{CL}, since $\ell^q$ has Clarkson roundness $\min(q,q' )$, an application of Theorem \ref{LB} shows that $\ell^{p}(\ell^q)$ has Clarkson roundness $\min(p,p' ,q,q' )$. Again, Lemma \ref{CRR} shows that it also has roundness $\min(p,p' ,q,q' )$. To see that this is the maximal roundness, note that $\ell^p(\ell^q)$ contains isometric copies of both $\ell^p$ and $\ell^q$.
\end{proof}

It is known that the Schatten $p$-trace class $S_{p}$ also satisfies Clarkson's Inequalities (see \cite{M}). More precisely, in our terminology we have the following theorem.

\begin{thm}
If $1<p<\infty$, then $S_{p}$ has Clarkson roundness $\min(p,p' )$.
\end{thm}

Thus, Lemma \ref{CRR} together with the fact that $S_{p}$ contains an isometric copy of $\ell^p$ allows us to compute the maximal roundess of $S_{p}$. This result was first proved in \cite{LTW}.

\begin{cor}
If $1\leq p<\infty$, then the maximal roundess of $S_{p}$ is $\min(p,p' )$.
\end{cor}

\section{Duality, smoothness and convexity}\label{Duality and Coroundness}

We now introduce the natural dual property of roundness. Note that this property appears at least implicitly in \cite{KT3} (see also \cite{KTY,THK}).

\begin{defn} For $2 \le p < \infty$ we shall say that a Banach space $X$ has \textit{coroundness $p$} if
   \[  \norm{x+y}^{p} + \norm{x-y}^{p} \le 2^{p-1} \bigl(\norm{x}^p + \norm{y}^p \bigr)
   \]
for all $x,y \in X$.
\end{defn}

While $X$ has roundness $p$ if and only if $\nu_X(p) = 2$ we have that $X$ has coroundness $p$ if and only if $\nu_{X}(p) = 2^{p-1}$. That is, both roundness and coroundness correspond to the situation where $\nu_X$ takes on its minimum possible value for the given value of $p$.

\begin{rem}\label{Type and Cotype}
	Roundness and coroundness are intimately related to the concepts of Rademacher type and cotype. To detail this relation, let us recall the definitions of Rademacher type and cotype.
	
	Let $1\leq p\leq 2$. A Banach space $X$ is said to be of \textit{Rademacher type $p$} if there exists some $K>0$ such that
		\[ \int_{0}^{1}\bigg\|\sum_{i=1}^{n}r_{i}(t)x_{i}\bigg\|^{p}dt\leq K\sum_{i=1}^{n}\|x_{i}\|^{p}
		\]
	for all $x_{1},\dots,x_{n}\in X$ and $n\geq 1$, where $r_{i}(t)$ are the Rademacher functions, that is, $r_{i}(t)=\mbox{sgn}(\sin(2^{i}\pi t))$. The smallest constant $K>0$ for which the above inequalities hold is denoted by $T_{p}(X)$.
	
	Let $2\leq q<\infty$. A Banach space $X$ is said to be of \textit{Rademacher cotype $q$} if there exists some $K>0$ such that
	\[ \sum_{i=1}^{n}\|x_{i}\|^{q}\leq K\int_{0}^{1}\bigg\|\sum_{i=1}^{n}r_{i}(t)x_{i}\bigg\|^{q}dt
	\]
	for all $x_{1},\dots,x_{n}\in X$ and $n\geq 1$. The smallest constant $K>0$ for which the above inequalities hold is denoted by $C_{q}(X)$.
	
	It follows from Theorems 2.2 and 2.4 in \cite{KT3} that if $1\leq p\leq 2$ and $2\leq q<\infty$ then a Banach space $X$ has Rademacher type $p$ if and only if it has type $p$ with $T_{p}(X)=1$ and Rademacher cotype $q$ if and only if it has cotype $q$ with $C_{q}(X)=1$. That is, roundness and coroundness are simply special cases of Rademacher type and cotype.
\end{rem}

\begin{figure}[ht!]
\begin{center}
\begin{tikzpicture}[xscale=3.0]

\draw[->,red] (1,1.5) -- (3,1.5);
\draw[->,red] (1,1.5) -- (1,7);
\draw[red] (3,1.5) node[below] {$p$};

\draw[red] (0.95,2) -- (1.05,2);
\draw (0.95,2) node[left] {$2$};

\draw[red] (1,1.6) -- (1,1.4);
\draw (1,1.4) node[below] {$1$};

\draw[red] (2,1.6) -- (2,1.4);
\draw (2,1.4) node[below] {$2$};

\draw[red] (1.6,1.6) -- (1.6,1.4);
\draw (1.6,1.4) node[below] {$\mr(X)$};

\draw[red] (2.667,1.6) -- (2.667,1.4);
\draw (2.667,1.4) node[below] {$\mc(X)$};

\draw[dashed,blue] (1,2) -- (2,2);
\draw[dashed,blue] (2., 2.) -- (2.03, 2.04) -- (2.07, 2.10) -- (2.10, 2.14) -- (2.13, 2.19) -- (2.17, 2.25) -- (2.20, 2.30) -- (2.23, 2.35) -- (2.27, 2.41) -- (2.30, 2.46) -- (2.33, 2.51) -- (2.37, 2.58) -- (2.40, 2.64) -- (2.43, 2.69) -- (2.47, 2.77) -- (2.50, 2.83) -- (2.53, 2.89) -- (2.57, 2.97) -- (2.60, 3.03) -- (2.63, 3.10) -- (2.67, 3.18) -- (2.70, 3.25) -- (2.73, 3.32) -- (2.77, 3.41) -- (2.80, 3.48) -- (2.83, 3.56) -- (2.87, 3.66) -- (2.90, 3.73) -- (2.93, 3.81) -- (2.97, 3.92) -- (3., 4.);
\draw (2.5,2.5) node[right] {$2^{p-1}$};
\draw[dashed,blue] (1., 2.) -- (1.03, 2.04) -- (1.07, 2.10) -- (1.10, 2.14) -- (1.13, 2.19) -- (1.17, 2.25) -- (1.20, 2.30) -- (1.23, 2.35) -- (1.27, 2.41) -- (1.30, 2.46) -- (1.33, 2.51) -- (1.37, 2.58) -- (1.40, 2.64) -- (1.43, 2.69) -- (1.47, 2.77) -- (1.50, 2.83) -- (1.53, 2.89) -- (1.57, 2.97) -- (1.60, 3.03) -- (1.63, 3.10) -- (1.67, 3.18) -- (1.70, 3.25) -- (1.73, 3.32) -- (1.77, 3.41) -- (1.80, 3.48) -- (1.83, 3.56) -- (1.87, 3.66) -- (1.90, 3.73) -- (1.93, 3.81) -- (1.97, 3.92) -- (2., 4.) -- (2.03, 4.08) -- (2.07, 4.20) -- (2.10, 4.29) -- (2.13, 4.38) -- (2.17, 4.50) -- (2.20, 4.59) -- (2.23, 4.69) -- (2.27, 4.82) -- (2.30, 4.92) -- (2.33, 5.03) -- (2.37, 5.17) -- (2.40, 5.28) -- (2.43, 5.39) -- (2.47, 5.54) -- (2.50, 5.66) -- (2.53, 5.78) -- (2.57, 5.94) -- (2.60, 6.06);
\draw (2,4.05) node[above] {$2^p$};

\draw[thick]
(1., 2.0) -- (1.03, 2.0) -- (1.07, 2.0) -- (1.10, 2.0) -- (1.13, 2.0) -- (1.17, 2.0) -- (1.20, 2.0) -- (1.23, 2.0) -- (1.27, 2.0) -- (1.30, 2.0) -- (1.33, 2.0) -- (1.37, 2.0) -- (1.40, 2.0) -- (1.43, 2.0) -- (1.47, 2.0) -- (1.50, 2.0) -- (1.53, 2.0) -- (1.57, 2.0) -- (1.60, 2.) -- (1.63, 2.0262) -- (1.67, 2.0617) -- (1.70, 2.0885) -- (1.73, 2.1158) -- (1.77, 2.1528) -- (1.80, 2.1810) -- (1.83, 2.2096) -- (1.87, 2.2482) -- (1.90, 2.2776) -- (1.93, 2.3073) -- (1.97, 2.3476) -- (2., 2.3784) -- (2.03, 2.4096) -- (2.07, 2.4517) -- (2.10, 2.4837) -- (2.13, 2.5161) -- (2.17, 2.5601) -- (2.20, 2.5937) -- (2.23, 2.6277) -- (2.27, 2.6736) -- (2.30, 2.7085) -- (2.33, 2.7438) -- (2.37, 2.7918) -- (2.40, 2.8284) -- (2.43, 2.8655) -- (2.47, 2.9156) -- (2.50, 2.9537) -- (2.53, 2.9922) -- (2.57, 3.0445) -- (2.60, 3.0844) -- (2.63, 3.1249) -- (2.67, 3.1821) -- (2.70, 3.2490) -- (2.73, 3.3173) -- (2.77, 3.4105) -- (2.80, 3.4822) -- (2.83, 3.5554) -- (2.87, 3.6553) -- (2.90, 3.7321) -- (2.93, 3.8106) -- (2.97, 3.9177) -- (3., 4.);

\draw (2,2.4) node[above] {$\nu_X(p)$};
\end{tikzpicture}
\caption{Roundness, coroundness and $\nu_X(p)$.}\label{nu-pic}
\end{center}
\end{figure}
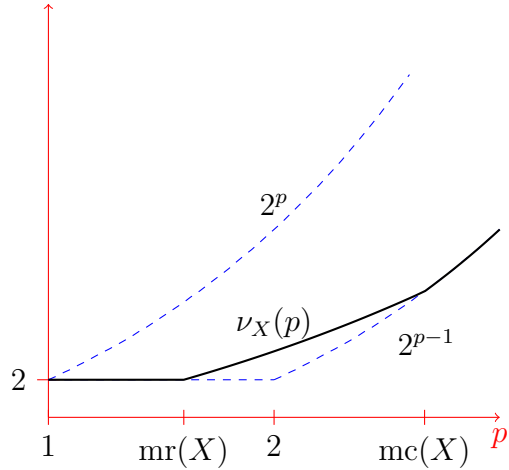

We now show that coroundness is in fact a property that is dual to roundness. In what follows, by convention we will say that every Banach space has coroundness $\infty$.

\begin{prop}\label{Duality}
If $1\leq p\leq\infty$ then $X$ has roundness $p$ if and only if $X^{\ast}$ has coroundness $p' $. Also, $X$ has coroundness $p$ if and only if $X^{\ast}$ has roundness $p' $.
\end{prop}
\begin{proof}
For $p=1,\infty$ the statement is trivial. Hence we may assume that $1<p<\infty$. In this case, the result follows from Lemma \ref{L2-1}. Indeed, if $\nu_{X}(p)=2$ then
	\[\nu_{X^{\ast}}(p' )=\nu_{X}(p)^{p' /p}=2^{p' -1}.
	\]
The other implications are proved similarly.
\end{proof}

\begin{cor}
If $2\leq p\leq p_{1}\leq\infty$ and $X$ has coroundness $p$ then $X$ also has coroundness $p_{1}$.
\end{cor}
\begin{proof}
Use duality via Proposition \ref{Duality}, together with the fact that if $X$ has roundness $q$ where $1\leq q_{1}\leq q$, then $X$ also has roundness $q_{1}$.
\end{proof}

It now makes sense to speak of the minimal coroundness that a Banach space may possess. We will denote this by $\mc(X)=\inf\{p:\nu_{X}(p)=2^{p-1}\}$ (where by convention we set $\mc(X)=\infty$ if this infimum is taken over the empty set). Note that $2\leq\mc(X)\leq\infty$. Clearly if $X$ embeds isometrically in $Y$ then $\mc(Y)\geq\mc(X)$. With this notation, Proposition \ref{Duality} gives the following.

\begin{cor}\label{Duality2}
The maximal roundness and minimal coroundness of a Banach space $X$ and its dual space $X^{\ast}$ satisfy $\mr(X)=\mc(X^{\ast})^{\prime}$ and $\mc(X)=\mr(X^{\ast})^{\prime}$. In particular, $\mr(X)=\mr(X^{\ast\ast})$ and $\mc(X)=\mc(X^{\ast\ast})$.
\end{cor}

\begin{rem}\label{Enflo Claim}
	In light of the above corollary it is a very natural question to ask whether or not one has $\mr(X)=\mr(X^{\ast})$ or $\mc(X)=\mc(X^{\ast})$ for all Banach spaces. In fact, when $X$ is $\ell^{p}$, $S_{p}$ or $\ell^{p}(\ell^{q})$ the results of Sections \ref{Roundness Of Standard Spaces} and \ref{Duality and Coroundness} show that these relations do indeed hold in these cases. However, it turns out that these relations do not hold in general. In fact, Enflo in \cite{E3} states without proof that it is possible to construct a $2$-dimensional Banach space whose maximal roundness is not equal to that of its dual. We shall exhibit an explicit example of such a space in Section \ref{Further Examples} below.
\end{rem}

The following well known geometric properties and functions (see for example \cite{LT}, Chapter 1 Section e) are intimately related to the study of roundness and coroundness of a Banach space.

\begin{defn}\label{Smoothness and Convexity} Let $X$ be a Banach space.
	\begin{itemize}
		\item[(i)]  $X$ is said to be \textit{uniformly smooth} if for all $\epsilon>0$ there exists some $\delta>0$ such that if $x,y\in X$ with $\norm{x}=1$ and $\norm{y}\leq\delta$ then
		\[ \norm{x+y}+\norm{x-y}\leq 2+\epsilon\norm{y}.
		\]
		\item[(ii)] $X$ is said to be \textit{uniformly convex} if for all $\epsilon>0$  there exists some $\delta>0$ such that if $x,y\in X$ with $\norm{x-y}\geq\epsilon$ then
		\[ \bigg\|\frac{x+y}{2}\bigg\|\leq 1-\delta.
		\]
		\item[(iii)] The \textit{modulus of uniform smoothness} of $X$ is the function $\rho_{X}:[0,\infty)\rightarrow\mathbb{R}$ defined by
		\[ \rho_{X}(t)=\sup\bigg\{\frac{1}{2}(\norm{x+y}+\norm{x-y})-1:\norm{x}=1,\norm{y}=t\bigg\}.
		\]
		\item[(iv)] The \textit{modulus of uniform convexity} of $X$ is the function $\delta_{X}:[0,2]\rightarrow[0,1]$ defined by
		\[ \delta_{X}(\epsilon)=\inf\bigg\{1-\frac{\norm{x+y}}{2}:\norm{x}=\norm{y}=1,\norm{x-y}\geq\epsilon\bigg\}.
		\]
	\end{itemize}
\end{defn}

We refer the reader to \cite{LT} for the following facts about uniformly smooth and uniformly convex spaces.

\begin{thm}\label{Omnibus} Let $X$ be a Banach space. Then
\begin{itemize}
\item[(i)] $X$ is uniformly smooth if and only if $\lim_{t\rightarrow 0^{+}}\rho_{X}(t)/t=0$.
\item[(ii)]  $X$ is uniformly convex if and only if $\delta_{X}(\epsilon)>0$ for all $\epsilon\in(0,2)$.
\item[(iii)\label{Reflexive}] if $X$ is uniformly smooth or uniformly convex then it is reflexive.
\item[(iv)]\label{SCD}  $X$ is uniformly smooth if and only if $X^{\ast}$ is uniformly convex. Also, by reflexivity, $X$ is uniformly convex if and only if $X^{\ast}$ is uniformly smooth.
\end{itemize}
\end{thm}

It is now possible to demonstrate that roundness is a smoothness property and that its dual property coroundess is a convexity property.

\begin{thm}\label{Geometry}
Let $X$ be a Banach space.
\begin{itemize}
	\item [(i)] If $\mr(X)>1$ then $X$ is uniformly smooth.
	\item [(ii)] If $\mc(X)<\infty$ then $X$ is uniformly convex.
\end{itemize}
\end{thm}
\begin{proof}
We begin by proving the second statement. So, suppose that $X$ has coroundness $p$ for some $2\leq p<\infty$. Then if $x,y\in X$ and $0<\epsilon<2$ with $\norm{x}=\norm{y}=1$ and $\norm{x-y}\geq\epsilon$ we have that
	\[ \norm{x+y}^p+\norm{x-y}^p\leq 2^{p-1}(\norm{x}^p+\norm{y}^p)=2^p
	\]
so that
	\[ \bigg\|\frac{x+y}{2}\bigg\|\leq(1-(\epsilon/2)^p)^{1/p}.
	\]
Hence by taking $\delta=1-(1-(\epsilon/2)^p)^{1/p}>0$ we see that $X$ is uniformly convex, which proves the second statement.

To prove the first statement, note that if $\mr(X)>1$ then Corollary \ref{Duality2} implies that $\mc(X^{\ast})<\infty$. By what was just proved this implies that $X^{\ast}$ is uniformly convex. But then by Theorem~\ref{Omnibus}(iv) we have that $X$ is uniformly smooth.
\end{proof}

Together with Theorem~\ref{Omnibus}(iii) this gives the following.

\begin{cor}
If $\mr(X)>1$ or $\mc(X)<\infty$ then $X$ is reflexive.
\end{cor}

Finally, we finish this section by noting that it is easy to compute the coroundness of all of those spaces whose roundness was computed in the previous section, by means of duality.

\begin{cor}
Let $1\leq p<\infty$. Then $\ell^p$ has minimal coroundness $\max(p,p' )$.
\end{cor}

\begin{cor}
	Let $1\leq p,q<\infty$. Then $\ell^{p}(\ell^{q})$ has minimal coroundness $\max(p,p' ,q,q' )$.
\end{cor}

\begin{cor}
	Let $1\leq p<\infty$. Then $S^p$ has minimal coroundness $\max(p,p' )$.
\end{cor}

\bigskip

\section{Characterisations of spaces with non-trivial roundness}\label{Nontrivial Roundness}

As stated earlier, a Banach space $X$ has $\mr(X)=2$ if and only if it is a Hilbert space. In this section we will focus on characterising those spaces whose maximal roundness is at the other extreme, that is, spaces for which $\mr(X)=1$. Actually, it turns out that it is easier to write down conditions which are equivalent to having $\mr(X)>1$ instead, and so this is what we will do.

Our starting point is a result of Enflo \cite{E4} which relates the study of roundness to the study of $p$-uniform smoothness, of which we now recall the definition (see for example \cite{KTY}).

\begin{defn}
	Let $X$ be a Banach space and $1\leq p\leq 2$. Then $X$ is said to be $p$-uniformly smooth if there exists $K>0$ such that
		\[ \|x+y\|^{p}+\|x-y\|^{p}\leq2(\|x\|^{p}+\|Ky\|^{p})
		\]
	for all $x,y\in X$.
\end{defn}

By taking $K=1$ above, it is clear that if a Banach space has roundness $p$ then it is also $p$-uniformly smooth. Enflo's result is the following.

\begin{lem}[\text{\cite[Proposition 2]{E4}}]\label{Uniform Smoothness Equiv}
	Let $X$ be a Banach space. Then $\mr(X)>1$ if and only if $X$ is $p$-uniformly smooth, for some $p>1$.
\end{lem}

Let $X$ be a Banach space, and let $\rho_{X}$ be the modulus of uniform smoothness of $X$. As in \cite{LT}, say that $\rho_{X}$ is of power type $p\geq 1$ if there exists some $C>0$ such that $\rho_{X}(t)\leq Ct^{p}$ for all $t\geq 0$. It is known that for $1\leq p\leq 2$, a Banach space $X$ is $p$-uniformly smooth if and only if $\rho_{X}$ is of power type $p$ (see for example \cite{KTY}). This fact together with Enflo's result gives the following characterisation of when $\mr(X)>1$.

\begin{thm}\label{Characterisation}
Let $X$ be a Banach space. Then $\mr(X) > 1$ if and only if $\rho_{X}$ is of power type $p$, for some $p>1$.
\end{thm}

\begin{rem}\label{Duality Equiv}
It is possible to use duality to translate the above results to the coroundness setting. For $2\leq q<\infty$, a Banach space $X$ is said to be $q$-uniformly convex if there exists $K>0$ such that
	\[ \|x+y\|^{p}+\|x-y\|^{p}\geq2(\|x\|^{p}+\|Ky\|^{p})
	\]
for all $x,y\in X$ (see for example \cite{KTY}). Letting $\delta_{X}$ denote the modulus of uniform convexity of $X$, we say that $\delta_{X}$ is of power type $q\geq2$ if there exists some $C>0$ such that $\delta_{X}(\epsilon)\geq C\epsilon^{q}$ for all $\epsilon\in[0,2)$. If $q\geq 2$ then $X$ is $q$-uniformly convex if and only if $\delta_{X}$ is of power type $q$. Similarly, $\rho_{X^{\ast}}$ is of power type $p>1$ if and only if $\delta_{X}$ is of power type $p' <\infty$ (see \cite{LT}). Since by Corollary~\ref{Duality2} we know that $\mr(X^{\ast})>1$ if and only if $\mc(X)<\infty$, Theorem~\ref{Characterisation} also gives the following characterisation of those spaces which possess finite coroundness: A Banach space $X$ has $\mc(X)<\infty$ if and only if $X$ is $q$-uniformly convex for some $q\geq 2$ if and only if $\delta_{X}$ is of power type $q$ for some $q\geq 2$.
\end{rem}

\begin{rem}
Theorem \ref{Characterisation} can be regarded as an isometric characterisation of those spaces for which $\mr(X)>1$. However, a classical theorem of Pisier \cite{P} states that a Banach space $X$ admits an equivalent renorming such that $\rho_X$ is of power type $p$, for some $p>1$, if and only if $X$ is super-reflexive. In this way, one obtains the following \textit{isomorphic} characterisation of those spaces for which $\mr(X)>1$: A Banach space $X$ is isomorphic to a Banach space $Y$ with $\mr(Y)>1$ if and only if $X$ is super-reflexive. Similarly of course, we find that a Banach space $X$ is isomorphic to a Banach space $Y$ with $\mc(Y)<\infty$ if and only if $X$ is super-reflexive. In particular, a Banach space $X$ is isomorphic to a Banach space $Y$ with $\mr(Y)>1$ if and only if $X$ is isomorphic to a Banach space $Z$ with $\mc(Z)<\infty$.
\end{rem}

For many of the results to follow, it will be helpful to use the notions of Fr\'{e}chet differentiability and support maps in Banach spaces. For more on these concepts, see \cite{D}.

\begin{defn}
	Let $X$ be a Banach space. We will say that $X$ is uniformly Fr\'{e}chet differentiable if the limits
		\[ \lim_{t\rightarrow 0}\frac{\norm{x+ty}-\norm{x}}{t}
		\]
	exist uniformly for all $x,y\in S_{X}$.
\end{defn}

\begin{defn}
	Let $X$ be a Banach space. A support map on $X$ is a map from $X-\{0\}$ to $X^{\ast}-\{0\}$, say $x\mapsto f_{x}$, such that the following two conditions hold.
	\begin{itemize}
		\item [(i)] For each $x\in S_{X}$, the functional $f_{x}$ has norm one and it norms $x$ in the sense that $f_{x}(x)=\norm{x}=1$.
		\item [(ii)] For each $x\neq 0$ and $\lambda >0$ we have that $f_{\lambda x}=\lambda f_{x}$.
	\end{itemize}
\end{defn}

The relationship between Fr\'{e}chet differentiability and support maps is explained by the following result (for a proof, see Theorem 1 on page 36 of \cite{D}, say).

\begin{thm}
	Let $X$ be a Banach space. If $X$ is uniformly smooth then there is a unique support map on $X$. Furthermore, $X$ is uniformly Fr\'{e}chet differentiable and if $x\mapsto f_{x}$ denotes the unique support map on $X$ then one has that
		\[ \lim_{t\rightarrow 0}\frac{\norm{x+ty}-\norm{x}}{t}=f_{x}(y)
		\]
	for all $x,y\in S_{X}$.
\end{thm}

Theorem \ref{Characterisation} allows us to say the following about uniform Frechet differentiability and non-trivial roundness.

\begin{prop}\label{Good Proposition}
Let $X$ be a Banach space and let $x\mapsto f_{x}$ be a support map on $X$. Suppose that there exists some $p>1$ such that
\[ \norm{x+ty}=1+f_{x}(y)t+O(t^{p})
\]
as $t\rightarrow 0$, where this $O(t^{p})$ holds uniformly over all $x,y\in S_{X}$. Then $\rho_{X}$ is of power type $p>1$ and consequently $\mr(X)>1$.
\end{prop}
\begin{proof}
	By assumption, there exists $C>0$ and $0<t_{0}\leq 1$ such that if $x,y\in S_{X}$ and $-t_{0}\leq t\leq t_{0}$ then
		\[ \big|\|x+ty\|-(1+f_{x}(y)t)\big|\leq Ct^{p}.
		\]
	Hence, for such $x,y,t$ one has that
		\begin{align*}
			\|x+ty\|&=\big|\|x+ty\|\big|
			\\&=\big|\|x+ty\|-(1+f_{x}(y)t)+(1+f_{x}(y)t)\big|
			\\&\leq\big|\|x+ty\|-(1+f_{x}(y)t)\big|+|1+f_{x}(y)t|
			\\&\leq Ct^{p}+|1+f_{x}(y)t|.
		\end{align*}
	Also, note that since $\|f_{x}\|=1$ one has that $|f_{x}(y)t|\leq\|f_{x}\|\|y\||t|\leq 1$. Hence
		\[ \|x+ty\|\leq Ct^{p}+|1+f_{x}(y)t|=Ct^{p}+1+f_{x}(y)t.
		\]
	Similarly, one can show that
		\[ \|x-ty\|\leq Ct^{p}+1+f_{x}(-y)t=Ct^{p}+1-f_{x}(y)t.
		\]
	Thus, if $x,y\in S_{X}$ and $0\leq t\leq t_{0}$ one has that
		\[ \frac{1}{2}(\|x+ty\|+\|x-ty\|)-1\leq\frac{1}{2}(Ct^{p}+1+f_{x}(y)t+Ct^{p}+1-f_{x}(y)t)-1=Ct^{p}.
		\]
	Hence $\rho_{X}(t)\leq Ct^{p}$, for all $0\leq t\leq t_{0}$, which shows that $\rho_{X}$ is of power type $p$. The fact that $\mr(X)>1$ now follows from Theorem \ref{Characterisation}.
\end{proof}

Our next aim is to provide a necessary condition for $\mr(X)$ to be greater than $1$. We require a number of preliminary lemmas.

\begin{lem}[see {\cite[Lemma 1.e.8]{LT}}]\label{Inverse}
If $X$ is uniformly convex, then $\delta_{X}:[0,2]\rightarrow[0,1]$ is strictly increasing and continuous with $\delta_{X}(0)=0$ and $\delta_{X}(2)=1$. Consequently $\delta_{X}$ has an inverse function $\delta_{X}^{-1}:[0,1]\rightarrow[0,2]$.
\end{lem}

\begin{lem}\label{Inverse Bound}
Suppose that $\rho_{X}$ is of power type $p$, for some $p>1$. Then there exists $M>0$ such that
	\[ \delta_{X^{\ast}}^{-1}(\epsilon)\leq M\epsilon^{1/p' }
	\]
for all $0\leq\epsilon\leq1$.
\end{lem}
\begin{proof}
By duality (see \cite{LT}), we have that $\delta_{X^{\ast}}$ is of power type $p'$. That is, there exists $M_{0}>0$ such that $\delta_{X^{\ast}}(\epsilon)\geq M_{0}\epsilon^{p'}$ for all $0\leq\epsilon\leq2$. By Theorem~\ref{Omnibus}(ii) this implies that $X^{\ast}$ is uniformly convex and hence by Lemma \ref{Inverse} we have that the inverse function $\delta_{X^{\ast}}^{-1}:[0,1]\rightarrow[0,2]$ exists. Putting this all together we find that for all $0\leq\epsilon\leq 1$ we have that
	\[ \epsilon=\delta_{X^{\ast}}(\delta_{X^{\ast}}^{-1}(\epsilon))\geq M_{0}\,\delta_{X^{\ast}}^{-1}(\epsilon)^{p' }.
	\]
Rearranging we find that
	\[ \delta_{X^{\ast}}^{-1}(\epsilon)\leq M_{0}^{-1/p'}\epsilon^{1/p'}
	\]
for all $0\leq\epsilon\leq 1$ and so the result follows by taking $M=M_{0}^{-1/p'}$.
\end{proof}


\begin{lem}[see {\cite[Theorem 1, page 36]{D}}]\label{First Support Lemma}
Suppose that $X^{\ast}$ is uniformly convex. Then any support map $x\mapsto f_{x}$ on $X$ is uniformly continuous on $S_{X}$. More precisely, for any $0<\epsilon\leq 2$ and $x,y\in S_{X}$, if $\norm{x-y}<2\delta_{X^{\ast}}(\epsilon)$ then $\norm{f_{x}-f_{y}}<\epsilon$.
\end{lem}

\begin{lem}\label{Final Lemma}
Let $X$ be a Banach space, and let $x\mapsto f_{x}$ be any support map on $X$. Suppose that $\rho_{X}$ is of power type $p$, for some $p>1$. Then
	\[ \norm{x+ty}=1+f_{x}(y)t+O(t^{1+1/p' })
	\]
as $t\rightarrow 0$, where this $O(t^{1+1/p' })$ holds uniformly over $x,y\in S_{X}$.
\end{lem}
\begin{proof}
It is proved on page $21$ of \cite{D} that for $t\in(0,1/2)$ and $x,y\in S_{X}$, one has that
\[ f_{x}(y)\leq\frac{\norm{x+ty}-\norm{x}}{t}\leq\frac{f_{x+ty}(y)}{\norm{x+ty}}
\]
and hence that
\[ \bigg| \frac{\norm{x+ty}-\norm{x}}{t}-f_{x}(y)\bigg|\leq\bigg|\frac{f_{x+ty}(y)}{\norm{x+ty}}-f_{x}(y)\bigg|.
\]
In fact, by setting $t'=-t$ and $y'=-y$ it follows that the above inequality actually holds for all $t\in(-1/2,1/2)-\{0\}$ and $x,y\in S_{X}$. Hence, for all $t\in(-1/2,1/2)-\{0\}$ and $x,y\in S_{X}$, the basic properties of the support map imply that
	\[ \bigg| \frac{\norm{x+ty}-\norm{x}}{t}-f_{x}(y)\bigg|\leq\bigg|\frac{f_{x+ty}(y)}{\norm{x+ty}}-f_{x}(y)\bigg|=|(f_{x+ty/\|x+ty\|}-f_{x})(y)|\leq\|f_{x+ty/\|x+ty\|}-f_{x}\|.
	\]
Now, for $t\in(-1/2,1/2)-\{0\}$ and $x,y\in S_{X}$ we note that the triangle inequality implies that
\begin{align*}
\bigg\|\frac{x+ty}{\|x+ty\|}-x\bigg\|&\leq\bigg\|\frac{x+ty}{\|x+ty\|}-(x+ty)\bigg\|+\|(x+ty)-x\|
\\&=\|x+ty\|\bigg|\frac{1}{\norm{x+ty}}-1\bigg|+|t|
\\&\leq(1+|t|)\bigg|\frac{1}{\norm{x+ty}}-1\bigg|+|t|
\\&\leq\frac{3}{2}\bigg|\frac{1}{\norm{x+ty}}-1\bigg|+|t|.
\end{align*}
We claim that if $t\in(-1/2,1/2)-\{0\}$ and $x,y\in S_{X}$ then
\[ \bigg|\frac{1}{\norm{x+ty}}-1\bigg|\leq2|t|.
\]
In this case the triangle inequality clearly implies that
\[ 0<1-|t|\leq \|x+ty\|\leq 1+|t|.
\]
Then, since the function $h(\alpha)=1/\alpha-1$ is decreasing for $\alpha>0$ this implies that
\[ \bigg|\frac{1}{\norm{x+ty}}-1\bigg|\leq\max(|h(1-|t|)|,|h(1+|t|)|)=\max\bigg(\frac{|t|}{1-|t|},\frac{|t|}{1+|t|}\bigg).
\]
Since $0\leq |t|\leq 1/2$ one finds that
\[ \frac{|t|}{1-|t|}\leq\frac{|t|}{1-1/2}=2|t|
\]
and
\[ \frac{|t|}{1+|t|}\leq\frac{|t|}{1+0}=|t|.
\]
Putting this all together gives that if $t\in(-1/2,1/2)-\{0\}$ and $x,y\in S_{X}$ then
\[ \bigg|\frac{1}{\norm{x+ty}}-1\bigg|\leq\max(2|t|,|t|)=2|t|
\]
which proves the claim. By what was proved earlier, this means that if $t\in(-1/2,1/2)-\{0\}$ and $x,y\in S_{X}$ then
	\[ \bigg\|\frac{x+ty}{\|x+ty\|}-x\bigg\|\leq\frac{3}{2}\times 2|t|+|t|=4|t|.
	\]
Now, by Lemma \ref{Inverse Bound} we have that $\delta_{X^{\ast}}^{-1}:[0,1]\rightarrow[0,2]$ exists and that there exists some $M>0$ such that
	\[ \delta_{X^{\ast}}^{-1}(\epsilon)\leq M\epsilon^{1/p' }
	\]
for all $0\leq\epsilon\leq1$. Hence if we now further assume that $t\in(-1/4,1/4)-\{0\}$ and $x,y\in S_{X}$ then $0\leq 4|t|\leq 1$ and so
	\[ \bigg\|\frac{x+ty}{\|x+ty\|}-x\bigg\|\leq 4|t|=\delta_{X^{\ast}}(\delta_{X^{\ast}}^{-1}(4|t|))<2\delta_{X^{\ast}}(\delta_{X^{\ast}}^{-1}(4|t|))
	\]
and so Lemma \ref{First Support Lemma} implies that (since both $(x+ty)/\|x+ty\|$ and $x$ are in $S_{X}$)
	\[ \|f_{x+ty/\|x+ty\|}-f_{x}\|\leq\delta^{-1}_{X^{\ast}}(4|t|)\leq M(4|t|)^{1/p'}=(4^{1/p'}M)|t|^{1/p'}.
	\]
Thus, if $t\in(-1/4,1/4)-\{0\}$ and $x,y\in S_{X}$ then
	\[ \bigg| \frac{\norm{x+ty}-\norm{x}}{t}-f_{x}(y)\bigg|\leq(4^{1/p'}M)|t|^{1/p'}
	\]
which after multiplying both sides by $|t|$ reads as (since $\|x\|=1$)
	\[ \big|\|x+ty\|-1-f_{x}(y)t\big|\leq(4^{1/p'}M)|t|^{1+1/p'}.
	\]
Since the constant $M$ does not depend on the choice of $x,y\in S_{X}$ this proves the result.
\end{proof}

We now have the following characterisation of when $\rho_{X}$ is of power type $p$, for some $p>1$.

\begin{thm}\label{Split}
	Let $X$ be a Banach space and let $x\mapsto f_{x}$ be a support map on $X$. Then the following are equivalent.
	\begin{itemize}
		\item[(i)] $\rho_{X}$ is of power type $p$, for some $p>1$.
		\item[(ii)] There exists some $r>1$ such that
		\[ \norm{x+ty}=1+f_{x}(y)t+O(t^{r})
		\]
		as $t\rightarrow 0$, where this $O(t^{r})$ holds uniformly over all $x,y\in S_{X}$.
	\end{itemize}
\end{thm}
\begin{proof}
	The fact that (i) implies (ii) follows from Lemma \ref{Final Lemma} by taking $r=1+1/p'>1$. The fact that (ii) implies (i) is was already proved in Proposition \ref{Good Proposition}.
\end{proof}

Finally, we summarise all of our characterisations of when $\mr(X)>1$.

\begin{cor}
	Let $X$ be a Banach space, and let $x\mapsto f_{x}$ be a support map on $X$. Then the following are equivalent in pairs.
	\begin{itemize}
		\item[(i)] $X$ has non-trivial maximal roundness (i.e. $\mr(X)>1$).
		\item[(ii)] $\rho_{X}$ is of power type $p$, for some $p>1$.
		\item[(iii)] $X$ is $s$-uniformly smooth, for some $s>1$.
		\item[(iv)] There exists some $r>1$ such that
		\[ \norm{x+ty}=1+f_{x}(y)t+O(t^{r})
		\]
		as $t\rightarrow 0$, where this $O(t^{r})$ holds uniformly over all $x,y\in S_{X}$.
		\item[(v)] $X^{\ast}$ has non-trivial minimal coroundness (i.e. $\mc(X^{\ast})<\infty$).
		\item[(vi)] $\delta_{X^{\ast}}$ is of power type $q$, for some $q\geq 2$.
		\item[(vii)] $X^{\ast}$ is $w$-uniformly convex, for some $w\geq 2$.
		
	\end{itemize}
\end{cor}
\begin{proof}
	The equivalence of (i),(ii) and (iii) is simply the content of Lemma \ref{Uniform Smoothness Equiv} and Theorem \ref{Characterisation}. Theorem \ref{Split} gives the equivalence of (ii) and (iv). Finally, Remark \ref{Duality Equiv} shows that (v), (vi) and (vii) are equivalent, and that each of these is equivalent to (ii).
\end{proof}

\section{Further Examples}\label{Further Examples}

In this section we aim to give some further examples of Banach spaces that possess non-trivial roundness or coroundness. Unlike the spaces that we encountered in Section \ref{Roundness Of Standard Spaces}, here we shall deal with spaces for which a precise calculation of their maximal roundness or minimal coroundness is more or less infeasible. Instead, we shall seek to employ the results of the previous section so that we may still conclude non-trivial information about the values of roundness and coroundness that these spaces possess.

For our first two examples, we will be looking at particular examples of Orlicz sequence spaces, which in some sense are a generalisation of $\ell^{p}$. Let us recall the necessary definitions here. For more on Orlicz spaces, the reader is referred to \cite{LT}.

\begin{defn}
	An \textit{Orlicz function} $\Phi:[0,\infty)\rightarrow\mathbb{R}$ is a continuous strictly increasing and convex function such that $\Phi(0)=0$ and $\lim_{t\rightarrow\infty}\Phi(t)=\infty$. To any Orlicz function $\Phi$ we associate the space $\ell_{\Phi}$ of all sequences of real numbers $x=(x_{1},x_{2},\dots)$ such that $\sum_{n=1}^{\infty}\Phi(|x_{n}|/k)<\infty$ for some $k>0$. When equipped with the \textit{Luxemburg norm} $\|\cdot\|_{\Phi}$ defined by
		\[ \|x\|_{\Phi}=\inf\bigg\{k>0:\sum_{n=1}^{\infty}\Phi(|x_{n}|/k)\leq 1\bigg\}
		\]
	the space $(\ell_{\Phi},\|\cdot\|_{\Phi})$ becomes a Banach space, which we shall refer to as the \textit{Orlicz sequence space generated by $\Phi$}.
\end{defn}

It is often useful to deal only with Orlicz functions that satisfy the following growth rate condition.

\begin{defn}
	Let $\Phi:[0,\infty)\rightarrow\mathbb{R}$ be an Orlicz function. Then $\Phi$ is said to satisfy the \textit{$\Delta_{2}$ condition at $0$} if there exist $t_{0},C>0$ such that $\Phi(2t)\leq C\Phi(t)$, for all $0\leq t\leq t_{0}$.
\end{defn}

If $\Phi$ is an Orlicz function that is differentiable in an interval of the form $(0,t_{0})$, for some $t_{0}>0$, there is a simple way to test whether or not $\Phi$ satisfies the $\Delta_{2}$ condition at $0$. Namely, in this case $\Phi$ satisfies the $\Delta_{2}$ condition at $0$ if and only if (see the comments after Proposition 4.a.5 in \cite{LT})
\begin{equation}\label{Delta 2 Test}
	\limsup_{t\rightarrow 0}\frac{t\Phi^{\prime}(t)}{\Phi(t)}<\infty.
\end{equation}

Another notion that we will require when dealing with Orlicz sequence spaces is the notion of equivalence of functions. Suppose that $A\subseteq\mathbb{R}$ and that $f,g:A\rightarrow\mathbb{R}$. We will say that $f$ and $g$ are \textit{equivalent} if there exist $k,K>0$ such that
	\[ K^{-1}g(k^{-1}t)\leq f(t)\leq Kg(kt)
	\]
for all $t\in A$. If $A$ contains an interval of the form $[0,t_{0}]$, for some $t_{0}>0$, we will say that $f$ and $g$ are \textit{equivalent at $0$} if there exists $k,K>0$ and $0<t_{1}\leq t_{0}$ such that
	\[ K^{-1}g(k^{-1}t)\leq f(t)\leq Kg(kt)
	\]
for all $0\leq t\leq t_{1}$.

For our first example, we will give provide an Orlicz sequence space (that is different from $\ell^{p}$) that possesses non-trivial roundness. To do so we will use the following result when estimating the modulus of uniform smoothness of an Orlicz space. This result is really a combination of Proposition 19 and Lemma 20 from \cite{F1}.

\begin{thm}\label{Orlicz Smoothness}
	Let $\Psi:[0,\infty)\rightarrow\mathbb{R}$ be an Orlicz function that satisfies the $\Delta_{2}$ condition at $0$. Define $\Phi:[0,\infty)\rightarrow\mathbb{R}$ by
		\[ \Phi(t)=\bigg(\int_{0}^{1}\frac{\Psi(s)}{s}ds\bigg)^{-1}\int_{0}^{t}\frac{\Psi(s)}{s}ds
		\]
	for all $t\geq 0$. Then $\Phi$ is an Orlicz function equivalent to $\Psi$. Moreover, for $X=\ell_{\Phi}$, there exists $C>0$ such that if $t\in(0,1]$ then the modulus of uniform smoothness of $X$ satisfies
		\[ \rho_{X}(t)\leq Ct^{2}\sup_{\substack{t\leq u\leq 1 \\ 0<v\leq 1}}\frac{\Phi(uv)}{u^{2}\Phi(v)}.
		\]
\end{thm}

\begin{exam} \textit{(An Orlicz space with non-trivial roundness).}
	Here we will consider an Orlicz space that is essentially derived from the function $t^{p}(1+|\log(t)|)$. However, there are quite a few technical obstacles that one must overcome to produce a usable estimate for the modulus of uniform smoothness of a given Orlicz space. Of course, the main tool that we will use to do so is Theorem \ref{Orlicz Smoothness}.
	
	So, let us fix $1<p<2$. It can be checked using elementary calculus that there exists $t_{0}\in(0,1)$ such that  if we define $\Psi_{0}:[0,1]\rightarrow\mathbb{R}$ by $\Psi_{0}(t)=t^{p}|\log(t_{0}t)|+(1-1/p)t^{p}$, then $\Psi_{0}$ is strictly increasing and convex on the interval $[0,1]$. It is now possible to define the Orlicz function $\Psi:[0,\infty)\rightarrow\mathbb{R}$ by
		\[ \Psi(t)=\begin{cases} \Psi_{0}(t),&0\leq t\leq 1, \\ \Psi_{0}(1)+(t-1)\Psi_{0}^{\prime}(1),&1\leq t<\infty,\end{cases}
		\]
	where here $\Psi_{0}^{\prime}(1)$ denotes the derivative of $\Psi_{0}$ at $t=1$. We now claim that $\Psi$ satisfies the $\Delta_{2}$ condition at $0$. Indeed, one has that
		\[ \limsup_{t\rightarrow 0}\frac{t\Psi^{\prime}(t)}{\Psi(t)}=\lim_{t\rightarrow 0}\frac{p|\log(t_{0}t)|+p-2}{|\log(t_{0}t)|+1-1/p}=p.
		\]
	By  (\ref{Delta 2 Test}), this is sufficient to conclude that $\Psi$ satisfies the $\Delta_{2}$ condition at $0$.
	
	Next we compute $\Phi$, as given in Theorem \ref{Orlicz Smoothness}. If $t\in[0,1]$, elementary integration techniques show that
		\[ \Phi(t)=c_{p}t^{p}(1+|\log(t_{0}t)|)
		\]
	where $c_{p}=(1+|\log(t_{0})|)^{-1}$. Thus by Theorem \ref{Orlicz Smoothness}, for $X=\ell_{\Phi}$ and $t\in[0,1]$ one has that (note that here we will be using the fact that $1<p<2$)
	
		\begin{align*}
			\rho_{X}(t)&\leq Ct^{2}\sup_{\substack{t\leq u\leq 1 \\ 0<v\leq 1}}\frac{\Phi(uv)}{u^{2}\Phi(v)}
			\\&=Ct^{2}\sup_{\substack{t\leq u\leq 1 \\ 0<v\leq 1}}u^{p-2}\bigg(1+\frac{|\log(u)|}{|\log(t_{0}v)|+1}\bigg)
			\\&=Ct^{2}\times t^{p-2}(1+c_{p}|\log(t)|)
			\\&=Ct^{p}(1+c_{p}|\log(t)|).
		\end{align*}
	But it is easily checked that (by using L'H\^{o}pital's rule for example) if $0<q<p$ then
		\[ \lim_{t\rightarrow 0}\frac{Ct^{p}(1+c_{p}|\log(t)|)}{t^{q}}=0.
		\]
	Hence, $\rho_{X}$ is of power type $q$, for all $0<q<p$. By Theorem \ref{Characterisation}, since $p>1$ this is sufficient to conclude that $\mr(X)=\mr(\ell_{\Phi})>1$.
\end{exam}

For our second example, we will provide an Orlicz sequence space (again, different from $\ell^{p}$) that possesses non-trivial coroundness. This time of course, we will require estimates pertaining to the modulus of uniform convexity of Orlicz spaces. The result that we shall use here is the following one (which can be found in the remarks after Lemma 1.e.8 in \cite{LT} say).

\begin{thm}\label{Orlicz Convexity}
	Let $\Phi$ be an Orlicz function and let $X=\ell_{\Phi}$. Suppose that the following conditions hold:
	\begin{enumerate}
		\item $\Phi$ is super-multiplicative. That is, there exists $c>0$ such that $\Phi(st)\geq c\,\Phi(s)\Phi(t)$, for all $0\leq s,t\leq 1$.
		\item $\Phi(t^{1/2})$ is equivalent to a convex function.
	\end{enumerate}
	Then $\delta_{X}$, the modulus of uniform convexity of $X$, is equivalent to $\Phi$ at $0$. That is, there exist $k,K>0$ and $t_{0}>0$ such that
	\[ K^{-1}\Phi(k^{-1}t)\leq\delta_{X}(t)\leq K\Phi(kt)
	\]
	for all $0\leq t\leq t_{0}$.
\end{thm}

\begin{exam} \textit{(An Orlicz space with non-trivial coroundness).}
	For our second example, we will take $p\geq 2$ and define the Orlicz function $\Phi:[0,\infty)\rightarrow\mathbb{R}$ by
	\[ \Phi(t)=\begin{cases} \frac{t^{p}}{1+|\log(t)|},& 0\leq t\leq 1, \\ t^{2p},& 1<t<\infty. \end{cases}
	\]
	One can check that $\Phi(t^{1/2})$ is convex on $[0,\infty)$.
	
	We now show that $\Phi$ is super-multiplicative, in the sense given in Theorem \ref{Orlicz Convexity}. So, take $0\leq s,t\leq 1$. We claim that $\Phi(st)\geq\Phi(s)\Phi(t)$. Indeed, one has that
		\begin{align*}
			\Phi(st)\geq\Phi(s)\Phi(t)&\iff\frac{s^{p}t^{p}}{1+|\log(st)|}\geq\bigg(\frac{s^{p}}{1+|\log(s)|}\bigg)\bigg(\frac{t^{p}}{1+|\log(t)|}\bigg)
			\\&\iff (1+|\log(s)|)(1+|\log(t)|)\geq1+|\log(st)|.
		\end{align*}
	However, this last inequality is true since
	\begin{align*}
		1+|\log(st)|&=1+|\log(s)+\log(t)|
		\\&\leq1+|\log(s)|+|\log(t)|
		\\&\leq1+|\log(s)|+|\log(t)|+|\log(s)||\log(t)|
		\\&=(1+|\log(s)|)(1+|\log(t)|)
	\end{align*}
	and so the claim is proved. Hence, by Theorem \ref{Orlicz Convexity} there exist $k,K>0$ and $t_{0}>0$ such that
	\[ K^{-1}\Phi(k^{-1}t)\leq\delta_{X}(t)\leq K\Phi(kt)
	\]
	for all $0\leq t\leq t_{0}$. Now,
		\[ \lim_{t\rightarrow 0}\frac{\Phi(t)}{t^{q}}=\begin{cases} \infty, & p<q<\infty, \\ 0, & q=p.\end{cases}
		\]
	But due to the equivalence of $\Phi$ and $\delta_{X}$ at $0$ this also means that
		\[ \lim_{t\rightarrow 0}\frac{\delta_{X}(t)}{t^{q}}=\begin{cases} \infty, & p<q<\infty, \\ 0, & q=p.\end{cases}
		\]
	Hence $\delta_{X}$ is of power type $q$ for precisely those $q$ such that $p<q<\infty$. Consequently, by Remark \ref{Duality Equiv} we may conclude that $\mc(X)<\infty$.
\end{exam}

Of course, arguing as above Theorem \ref{Orlicz Convexity} gives the following sufficient conditions on $\Phi$ for the Orlicz space $\ell_{\Phi}$ to have non-trivial coroundness.

\begin{cor}
	Let $\Phi$ be an Orlicz function and let $X=\ell_{\Phi}$. Suppose that the following conditions hold:
	\begin{enumerate}
		\item $\Phi$ is super-multiplicative. That is, there exists $c>0$ such that $\Phi(st)\geq c\,\Phi(s)\Phi(t)$, for all $0\leq s,t\leq 1$.
		\item $\Phi(t^{1/2})$ is equivalent to a convex function.
		\item $\Phi$ is of power type $q$, for some $q\geq 2$. That is, there exist $C>0$ and $t_{0}>0$ such that $\Phi(t)\geq Ct^{q}$, for all $0\leq t\leq t_{0}$.
	\end{enumerate}
	Then $\mc(X)<\infty$.
\end{cor}

We now move onto our final example. As mentioned in Remark \ref{Enflo Claim}, Enflo in \cite{E3} claimed without proof that it should be possible to construct a Banach space $X$ whose maximal roundness was different from that of its dual. Moreover, he claimed that this should also be possible with $\dim(X)=2$. Here we confirm Enflo's claim, by explicitly providing a $2$-dimensional Banach space and using the results from Section \ref{Nontrivial Roundness} to prove that $\mr(X)\neq\mr(X^{\ast})$. To this end, let $R=\mathbb{R}^{2}$ and equip $R$ with the norm
	\[ \norm{(x_{1},x_{2})}_R = \begin{cases}
		|x_{2}|,     & \text{$|x_{2}| \ge |x_{1}|$}, \\
		\frac{x_{1}^{2}+x_{2}^{2}}{2|x_{1}|}, & \text{otherwise.}
		\end{cases}
	\]
It is a simple matter to check that this formula does indeed define a norm on $R=\mathbb{R}^{2}$. We shall refer to the space $(R,\|\cdot\|_{R})$ as the racetrack space. It is straightforward to verify that the dual space $R^{\ast}$ has dual norm given by
	\[ \|(x_{1},x_{2})\|_{R^{\ast}}=\sqrt{x_{1}^{2}+x_{2}^{2}}+|x_{1}|.
	\]
A picture of the unit spheres of the racetrack space and its dual, $S_{R}$ and $S_{R^{\ast}}$ respectively, is shown in Figure \ref{racetrack}.

 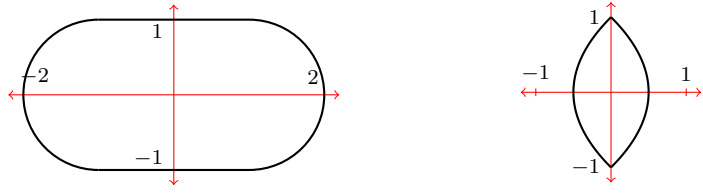
\begin{figure}
\begin{center}
 \begin{tikzpicture}
   \draw[<->,red] (-2.2,0) -- (2.2,0);
   \draw[<->,red] (0,-1.2) -- (0,1.2);
   \draw[thick] (-1,1) -- (1,1);
   \draw[thick] (-1,-1) -- (1,-1);
   \draw[thick] (-1,1) arc (90:270:1);
   \draw[thick] (1,-1) arc (-90:90:1);
   \draw (-1.85,0) node[above] {\tiny{$-2$}};
   \draw (1.85,0) node[above] {\tiny{$2$}};
   \draw (0,0.85) node[left] {\tiny{$1$}};
   \draw (0,-0.85) node[left] {\tiny{$-1$}};
 \end{tikzpicture}
  \hspace{2cm}
   \begin{tikzpicture}
   \draw[<->,red] (-1.2,0) -- (1.2,0);
   \draw[<->,red] (0,-1.2) -- (0,1.2);
   \draw[red] (-1,-0.05) -- (-1,0.05);
   \draw[red] (1,-0.05) -- (1,0.05);
   \draw[thick] plot[smooth,domain=-1:1] ({0.5*(1-\x*\x)},\x);
   \draw[thick] plot[smooth,domain=-1:1] ({-0.5*(1-\x*\x)},\x);
   \draw (-1,0) node[above] {\tiny{$-1$}};
   \draw (1,0) node[above] {\tiny{$1$}};
   \draw (0,1) node[left] {\tiny{$1$}};
   \draw (0,-1) node[left] {\tiny{$-1$}};
 \end{tikzpicture}
 \caption{The unit spheres of the racetrack space and its dual.}\label{racetrack}
\end{center}
 \end{figure}

From these pictures, one is led to suspect that $R$ is uniformly smooth but not uniformly convex, and that $R^{\ast}$ is uniformly convex but not uniformly smooth. Indeed, this turns out to be the case and we shall prove this by computing the asymptotic behaviour of the modulus of uniform smoothness of both $R$ and $R^{\ast}$. We start with the easier of the two computations.

\begin{prop}
	Let $\rho_{R^{\ast}}$ denote the modulus of uniform smoothness of the dual of the racetrack space. Then $\rho_{R^{\ast}}(t)=\Theta(t)$, as $t\rightarrow 0$.
\end{prop}
\begin{proof}
	Note that by the Triangle Inequality one trivially has that $\rho_{R^{\ast}}(t)\leq t$, for all $0\leq t\leq 1$. Now, keeping $0\leq t\leq 1$ let us set $x=(0,1)$ and $y=(t/2,0)$. Using the formula for $\|\cdot\|_{R^{\ast}}$ it is readily checked that $\|x\|_{R^{\ast}}=1$ and $\|y\|_{R^{\ast}}=t$. Hence
		\begin{align*}
			\rho_{R^{\ast}}(t)&\geq\frac{1}{2}(\|x+y\|_{R^{\ast}}+\|x-y\|_{R^{\ast}})-1
			\\&=\sqrt{\frac{t^{2}}{4}+1}+\frac{t}{2}-1
			\\&\geq 1+\frac{t}{2}-1
			 =\frac{t}{2}
		\end{align*}
	which completes the proof.
\end{proof}

Next, we have the much more tedious analogous computation for $R$. For this computation, we will use the following elementary inequality.

\begin{lem}\label{Simple Bound}
	Let $0<a\leq z\leq b$. Then
		$ \ds \frac{z^{2}+a^{2}}{2z}\leq\frac{b^{2}+a^{2}}{2b}
		$.
\end{lem}

\begin{prop}
	Let $\rho_{R}$ be the modulus of uniform smoothness of the racetrack space. Then $\rho_{R}(t)=\Theta(t^{2})$, as $t\rightarrow 0$. More precisely, if $0\leq t\leq 1/6$ then
		\[ (1+t^{2})^{1/2}-1\leq\rho_{R}(t)\leq18 t^{2}.
		\]
\end{prop}
\begin{proof}
	First, let us comment on the lower bound. This is due to a result of Nordlander \cite{N} who showed that if $X$ is any Banach space of dimension at least $2$ then
		\[ (1+t^{2})^{1/2}-1\leq\rho_{X}(t)
		\]
	for all $0\leq t\leq 1$. Note that
		\[ \lim_{t\rightarrow 0}\frac{(1+t^{2})^{1/2}-1}{t^{2}}=\frac{1}{2}
		\]
	which shows that $(1+t^{2})^{1/2}-1=\Theta(t^{2})$, as $t\rightarrow 0$.
	
	Now we turn our attention to the upper bound. Let
		\[ L=\{(s,1)\in\mathbb{R}^{2}:-1\leq s\leq 1\}
		\]
	and let
		\[ C=\{(\cos\theta+1,\sin\theta)\in\mathbb{R}^{2}:-\pi/2\leq\theta\leq\pi/2\}.
		\]
	Since the unit sphere of $R$ is given by $S_{R}=L\cup C\cup(-L)\cup(-C)$, by symmetry one has that if $t\neq 0$ then
		\[ \rho_{X}(t)=\sup\bigg\{\frac{1}{2}(\norm{x+y}+\norm{x-y})-1:x,y/t\in L\cup C\bigg\}.
		\]
	Also, let us put
		\[ S_{1}=\{(x_{1},x_{2})\in\mathbb{R}^{2}:|x_{2}|\geq|x_{1}|\}
		\]
	and $S_{2}=\mathbb{R}^{2}\setminus S_{1}$. We now consider four cases, depending on whether $x$ and $y/t$ are in $L$ or $C$. Also, for the remainder of the proof, we shall assume that $0<t\leq 1/6$.
	
	\textbf{Case 1.} Suppose that $x,y/t\in L$. That is, $x=(s,1)$ for some $-1\leq s\leq 1$ and $y=(tr,t)$ for some $-1\leq r\leq 1$. Then
	\[ x+y=(s+tr,1+t),\,x-y=(s-tr,1-t).
	\]
	Note that it is not possible to have $x+y\in S_{2}$ so there are only two cases to check.
	
	(1a) Suppose that $x+y,x-y\in S_{1}$. Then
	\[ \frac{1}{2}(\|x+y\|_{R}+\|x-y\|_{R})-1=\frac{1}{2}(|1+t|+|1-t|)-1=\frac{1}{2}((1+t)+(1-t))-1=0.
	\]
	(1b) Suppose that $x+y\in S_{1}$ and $x-y\in S_{2}$. This means that
	\[ |1+t|\geq|s+tr|,\,|s-tr|>|1-t|.
	\]
	In particular, the Triangle Inequality implies that
	\[ |1-t|\leq|s-tr|\leq|s|+|r|t\leq 1+t.
	\]
	Hence, applying Lemma \ref{Simple Bound} with $a=|1-t|=1-t,z=|s-tr|,b=|1+t|=1+t$ we have that
	\begin{align*}
		\frac{1}{2}(\|x+y\|_{R}+\|x-y\|_{R})-1&=\frac{1}{2}\bigg(|1+t|+\frac{(s-tr)^{2}+(1-t)^{2}}{2|s-tr|}\bigg)-1
		\\&\leq\frac{1}{2}\bigg(1+t+\frac{(1+t)^{2}+(1-t)^{2}}{2(1+t)}\bigg)-1
		\\&=\frac{t^{2}}{1+t}.
	\end{align*}
	Since $0<t\leq 1/6$, this implies that
	\[ \frac{1}{2}(\|x+y\|_{R}+\|x-y\|_{R})-1\leq\frac{t^{2}}{1+0}=t^{2}.
	\]
	
	\textbf{Case 2.} Suppose that $x\in L$ and $y/t\in C$. That is, $x=(s,1)$ for some $-1\leq s\leq 1$ and $y=(t\cos\theta+t,t\sin\theta)$ for some $-\pi/2\leq\theta\leq\pi/2$. Then
	\[ x+y=(s+t\cos\theta+t,1+t\sin\theta),\,x-y=(s-t\cos\theta-t,1-t\sin\theta).
	\]
	It is easy to check that since $0<t\leq1/6$ one cannot have both $x+y$ and $x-y\in S_{2}$, so there are only three cases to check.
	
	(2a) Suppose that $x+y,x-y\in S_{1}$. Then
	\[ \frac{1}{2}(\|x+y\|_{R}+\|x-y\|_{R})-1=\frac{1}{2}(|1+t\sin\theta|+|1-t\sin\theta|)-1=\frac{1}{2}((1+t\sin\theta)+(1-t\sin\theta))-1=0.
	\]
	(2b) Suppose that $x+y\in S_{1}$ and $x-y\in S_{2}$. This means that
	\[ |1+t\sin\theta|\geq|s+t\cos\theta+t|,\,|s-t\cos\theta-t|>|1-t\sin\theta|.
	\]
	In particular, the Triangle Inequality implies that
	\begin{align*}
		1-t\sin\theta&=|1-t\sin\theta|
		\\&\leq|s-t\cos\theta-t|
		\\&\leq|s+t\cos\theta+t|+2|t\cos\theta+t|
		\\&\leq|1+t\sin\theta|+2|t\cos\theta+t|
		\\&=1+\alpha t
	\end{align*}
	where here $\alpha=\sin\theta+2\cos\theta+2$. Hence, applying Lemma \ref{Simple Bound} with $a=|1-t\sin\theta|=1-t\sin\theta,z=|s-t\cos\theta-t|,b=1+\alpha t$ we have that
	\begin{align*}
		\frac{1}{2}(\|x+y\|_{R}+\|x-y\|_{R})-1&=\frac{1}{2}\bigg(|1+t\sin\theta|+\frac{(s-t\cos\theta-t)^{2}+(1-t\sin\theta)^{2}}{2|s-t\cos\theta-t|}\bigg)-1
		\\&\leq\frac{1}{2}\bigg(1+t\sin\theta+\frac{(1+\alpha t)^{2}+(1-t\sin\theta)^{2}}{2(1+\alpha t)}\bigg)-1
		\\&=\frac{(\alpha^{2}+\sin^{2}\theta+2\alpha\sin\theta)t^{2}}{4(1+\alpha t)}.
	\end{align*}
	Now, note that since $-\pi/2\leq\theta\leq\pi/2$ one has that
	\[ 0\leq\alpha=\sin\theta+2\cos\theta+2\leq 1+2+2=5.
	\]
	Hence
	\[ \frac{1}{2}(\|x+y\|_{R}+\|x-y\|_{R})-1\leq\frac{(5^{2}+1+2\times 5\times 1)t^{2}}{4}=9t^{2}.
	\]
	(2c) Suppose that $x+y\in S_{2}$ and $x-y\in S_{1}$. One may follow a similar argument to the above to show that again
	\[ \frac{1}{2}(\|x+y\|_{R}+\|x-y\|_{R})-1\leq9t^{2}.
	\]
	
	\textbf{Case 3.} Suppose that $x\in C$ and $y/t\in L$. That is, $x=(\cos\theta+1,\sin\theta)$ for some $-\pi/2\leq\theta\leq\pi/2$ and $y=(tr,t)$ for some $-1\leq r\leq 1$. Then
	\[ x+y=(\cos\theta+1+tr,\sin\theta+t),\,x-y=(\cos\theta+1-tr,\sin\theta-t).
	\]
	(3a) Suppose that $x+y,x-y\in S_{1}$. Since $0<t\leq1/6$, it is not hard to check that this is only possible for $\theta=\pm\pi/2$. If $\theta=\pi/2$ the fact that $x+y,x-y\in S_{1}$ implies that one must have $r=1$, and if $\theta=-\pi/2$ the fact that $x+y,x-y\in S_{1}$ implies that $r=-1$. But for $(\theta,r)=(\pi/2,1)$ and $(\theta,r)=(-\pi/2,-1)$ it is easily checked that
	\[ \frac{1}{2}(\|x+y\|_{R}+\|x-y\|_{R})-1=0.
	\]
	
	(3b) Suppose that $x+y\in S_{1}$ and $x-y\in S_{2}$. This means that
	\[ |\sin\theta+t|\geq|\cos\theta+1+tr|,\,|\cos\theta+1-tr|>|\sin\theta-t|.
	\]
	In particular, since $0<t\leq 1/6$, $-1\leq r\leq 1$ and $-\pi/2\leq\theta\leq\pi/2$, the first of these inequalities implies that either $\theta\geq 0$ or $\theta=-\pi/2$ and $r=-1$. But in the latter case we would have $x-y\notin S_{2}$. Hence we may assume from now on that $\theta\geq 0$. In this case, note that
	\[ |\sin\theta-t|\leq|\cos\theta+1-tr|\leq|\cos\theta+t+tr|+2t|r|\leq|\sin\theta+t|+2t\leq 1+3t.
	\]
	Also, since $\theta\geq 0$, we have that $(\sin\theta-t)^{2}\leq(1-t)^{2}$. Hence, applying Lemma \ref{Simple Bound} with $a=|\sin\theta-t|,z=|\cos\theta+1-tr|,b=1+3t$ we have that
	\begin{align*}
		\frac{1}{2}(\|x+y\|_{R}+\|x-y\|_{R})-1&=\frac{1}{2}\bigg(|\sin\theta+t|+\frac{(\cos\theta+t-tr)^{2}+(\sin\theta-t)^{2}}{2|\cos\theta+1-tr|}\bigg)-1
		\\&\leq\frac{1}{2}\bigg(1+t+\frac{(1+3t)^{2}+(1-t)^{2}}{2(1+3t)}\bigg)-1
		\\&=\frac{4t^{2}}{1+3t}
		\\&\leq\frac{4t^{2}}{1+3\times0}
		\\&=4t^{2}.
	\end{align*}
	
	(3c) Suppose that $x+y\in S_{2}$ and $x-y\in S_{1}$. One may follow a similar argument to the above to show that again
		\[ \frac{1}{2}(\|x+y\|_{R}+\|x-y\|_{R})-1\leq4t^{2}.
		\]
	
	(3d) Suppose that $x+y,x-y\in S_{2}$. This means that
	\[ |\cos\theta+1+tr|>|\sin\theta+t|,\,|\cos\theta+1-tr|>|\sin\theta-t|.
	\]
	Hence
	\begin{align*}
		&\frac{1}{2}(\|x+y\|_{R}+\|x-y\|_{R})-1
		\\&\quad\quad=\frac{1}{2}\bigg(\frac{(\cos\theta+1+tr)^{2}+(\sin\theta+t)^{2}}{2|\cos\theta+1+tr|}+\frac{(\cos\theta+1-tr)^{2}+(\sin\theta-t)^{2}}{2|\cos\theta+1-tr|}\bigg)-1
		\\&\quad\quad=\frac{1}{2}\bigg(\frac{(\cos\theta+1+tr)^{2}+(\sin\theta+t)^{2}}{2(\cos\theta+1+tr)}+\frac{(\cos\theta+1-tr)^{2}+(\sin\theta-t)^{2}}{2(\cos\theta+1-tr)}\bigg)-1
		\\&\quad\quad=\frac{(r(r-r\cos\theta-2\sin\theta)+\cos\theta+1)t^{2}}{2((\cos\theta+1)^{2}-t^{2}r^{2})}
		\\&\quad\quad\leq\frac{((1+1+2)+1+1)t^{2}}{2(1^{2}-t^{2})}
		\\&\quad\quad=\frac{3t^{2}}{1-t^{2}}
		\\&\quad\quad\leq\frac{3t^{2}}{1-(1/6)^{2}}
		\\&\quad\quad=\frac{108}{35}t^{2}.
	\end{align*}
	
	\textbf{Case 4.} Suppose that $x,y/t\in C$. That is, $x=(\cos\theta+1,\sin\theta)$ for some $-\pi/2\leq\theta\leq\pi/2$ and $y=(t\cos\phi+t,t\sin\phi)$ for some $-\pi/2\leq\phi\leq\pi/2$. Then
	\[ x+y=(\cos\theta+1+t\cos\phi+t,\sin\theta+t\sin\phi),\,x-y=(\cos\theta+1-t\cos\phi-t,\sin\theta-t\sin\phi).
	\]
	Note that it is not possible to have both $x+y\in S_{1}$ and $x-y\in S_{2}$ so there are only three cases to check.
	
	(4a) Suppose that $x+y,x-y\in S_{1}$. It is easy to check that it must be the case that $(\theta,\phi)=(\pi/2,\pi/2)$ or $(\theta,\phi)=(-\pi/2,-\pi/2)$. However, in both these cases
	\[ \frac{1}{2}(\|x+y\|_{R}+\|x-y\|_{R})-1=0.
	\]
	
	(4b) Suppose that $x+y\in S_{2}$ and $x-y\in S_{1}$. This means that
	\[ |\cos\theta+1+t\cos\phi+t|>|\sin\theta+t\sin\phi|,\,|\sin\theta-t\sin\phi|\geq|\cos\theta+1-t\cos\phi-t|.
	\]
	In particular this implies that
	\[ |\cos\theta+1+t\cos\phi+1|\leq|\cos\theta+1-t\cos\phi-t|+2t|\cos\phi+1|\leq|\sin\theta-t\sin\phi|+2t(\cos\phi+1).
	\]
	Hence, applying Lemma \ref{Simple Bound} with $a=|\sin\theta+t\sin\phi|,z=|\cos\theta+1+t\cos\phi+1|,b=|\sin\theta-t\sin\phi|+2t(\cos\phi+1)$ we have that
	
	\begin{align*}
		&\frac{1}{2}(\|x+y\|_{R}+\|x-y\|_{R})-1
		\\&\quad\quad=\frac{1}{2}\bigg(|\sin\theta-t\sin\phi|+\frac{(\cos\theta+1+t\cos\phi+t)^{2}+(\sin\theta+t\sin\phi)^{2}}{2|\cos\theta+1+t\cos\phi+t|}\bigg)-1
		\\&\quad\quad\leq\frac{1}{2}\bigg(|\sin\theta-t\sin\phi|+\frac{(|\sin\theta-t\sin\phi|+2t(\cos\phi+1))^{2}+(\sin\theta+t\sin\phi)^{2}}{2(|\sin\theta-t\sin\phi|+2t(\cos\phi+1))}\bigg)-1.
	\end{align*}
	
	Now, also note that the inequalities in this case and the Triangle Inequality imply that
	\[ 1-2t\leq\cos\theta+1-t(\cos\phi+1)=|\cos\theta+1-t\cos\phi-t|\leq|\sin\theta-t\sin\phi|\leq|\sin\theta|+t.
	\]
	Hence
	\[ |\sin\theta|\geq 1-3t\geq1-3\times\frac{1}{6}=\frac{1}{2}.
	\]
	So, suppose that $1/2\leq\sin\theta\leq 1$. In this case, one has that $|\sin\theta-t\sin\phi|=\sin\theta-t\sin\phi$. Hence if we put $\gamma=-\sin\phi+2(\cos\phi+1)$ then we have that
	\begin{align*}
		&\frac{1}{2}(\|x+y\|_{R}+\|x-y\|_{R})-1
		\\&\quad\quad\leq\frac{1}{2}\bigg(\sin\theta-t\sin\phi+\frac{(\sin\theta+\gamma t)^{2}+(\sin\theta+t\sin\phi)^{2}}{2(\sin\theta+\gamma t)}\bigg)-1
		\\&\quad\quad=\frac{(-2\sin\phi\gamma+\gamma^{2}+\sin^{2}\phi)t^{2}+4\gamma(\sin\theta-1)t+4\sin\theta(\sin\theta-1)}{4(\sin\theta+\gamma t)}.
	\end{align*}
	Note that
	\[ 0\leq\gamma=-\sin\phi+2(\cos\phi+1)\leq1+2\times 2=5.
	\]
	Using this and the fact that $1/2\leq\sin\theta\leq 1$ we find that
	\begin{align*}
		&\frac{1}{2}(\|x+y\|_{R}+\|x-y\|_{R})-1
		\\&\quad\quad\leq\frac{(2\times 5+5^{2}+1)t^{2}+0+0}{4\times1/2+0}
		\\&\quad\quad=18t^{2}.
	\end{align*}
	Similarly, it can be shown that if $-1\leq\sin\theta\leq-1/2$ then
		\[ \frac{1}{2}(\|x+y\|_{R}+\|x-y\|_{R})-1\leq18t^{2}
		\]
	
	(4c) Suppose that $x+y,x-y\in S_{2}$. This means that
	\[ |\cos\theta+1+t\cos\phi+t|>|\sin\theta+t\sin\phi|,\,|\cos\theta+1-t\cos\phi-t|>|\sin\theta-t\sin\phi|.
	\]
	Hence
	\begin{align*}
		&\frac{1}{2}(\|x+y\|_{R}+\|x-y\|_{R})-1
		\\&\quad=\frac{1}{2}\bigg(\frac{(\cos\theta+t\cos\phi+1+t)^{2}+(\sin\theta+t\sin\phi)^{2}}{2|\cos\theta+t\cos\phi+1+t|}
		\\&\quad\quad\quad+\frac{(\cos\theta-t\cos\phi+1-t)^{2}+(\sin\theta-t\sin\phi)^{2}}{2|\cos\theta-t\cos\phi+1-t|}\bigg)-1
		\\&\quad=\frac{1}{2}\bigg(\frac{(\cos\theta+t\cos\phi+1+t)^{2}+(\sin\theta+t\sin\phi)^{2}}{2(\cos\theta+t\cos\phi+1+t)}
		\\&\quad\quad\quad+\frac{(\cos\theta-t\cos\phi+1-t)^{2}+(\sin\theta-t\sin\phi)^{2}}{2(\cos\theta-t\cos\phi+1-t)}\bigg)-1
		\\&\quad=\frac{(-(\cos\theta+1)(\cos\phi+1)^{2}+(\cos\theta+1)\sin^{2}\theta+2(\cos\phi+1)^{2}-2\sin\theta\sin\phi(\cos\theta+1))t^{2}}{2((\cos\theta+1)^{2}-t^{2}(\cos\phi+1)^{2})}
		\\&\quad\leq\frac{(0+2\times 1^{2}+2\times2^{2}+2\times 2)t^{2}}{2(1^{2}-2^{2}\times t^{2})}
		\\&\quad=\frac{5t^{2}}{1-4t^{2}}
		\\&\quad\leq\frac{5t^{2}}{1-4\times(1/6)^{2}}
		\\&\quad=\frac{63}{8}t^{2}.
	\end{align*}
\end{proof}

By Theorem \ref{Characterisation}, the above two propositions allow us to conclude the following about the maximal roundness of $R$ and $R^{\ast}$.

\begin{cor}
	Let $R$ be the racetrack space. Then $\mr(R)>1$ and $\mr(R^{\ast})=1$. In particular, $\mr(R)\neq\mr(R^{\ast})$.
\end{cor}

\begin{rem}
	Finally, we remark that the above example can also be used to show that the minimal coroundness of a space is not always equal to the minimal coroundness of its dual. Indeed, in light of the above corollary, an application of Proposition \ref{Duality} shows that $\mc(R)=\infty$ and $\mc(R^{\ast})<\infty$, and hence $\mc(R)\neq\mc(R^{\ast})$.
\end{rem}

\section*{Acknowledgments}

The first author was partially supported by the grant from IPM (No. 96470412). The work of the third author was supported by the Research Training Program of the Department of Education and Training of the Australian Government.

\end{document}